\documentclass[12pt]{article}
\usepackage[utf8]{inputenc}
\usepackage[T1]{fontenc}
\usepackage{amsmath,amssymb,amsthm}
\usepackage{mathrsfs}
\usepackage[danish,english]{babel}
\usepackage[T1]{fontenc}
\usepackage[pass]{geometry}
\usepackage{graphicx}
\usepackage[usenames,dvipsnames]{color}
\usepackage{eucal}
\usepackage[all]{xy}
\usepackage{tikz}

\usetikzlibrary{matrix}

\usepackage{latexsym}
\usepackage{lipsum}
\usepackage{color}
\usepackage{accents}

\usepackage{hyperref}
\usepackage[capitalize]{cleveref}

\theoremstyle{theorem}
	\newtheorem{thm}{Theorem}[section]
	\newtheorem{lem}[thm]{Lemma}
	\newtheorem{cor}[thm]{Corollary}
	\newtheorem{prop}[thm]{Proposition}
	\newtheorem*{theorem*}{Theorem}
        	 \newtheorem*{prop*}{Proposition}	
	
\theoremstyle{definition}

\theoremstyle{remark}
	\newtheorem{remark}[thm]{Remark}

\newcommand{\F}{\mathbb{F}}
\newcommand{\N}{\mathbb{N}}
\newcommand{\Z}{\mathbb{Z}}
\newcommand{\p}{\mathcal{P}}

\renewcommand{\epsilon}{\varepsilon}
\renewcommand{\phi}{\varphi}

\newcommand{\ac}{\curvearrowright}

\newcommand{\Aut}{{ \rm Aut}}
\newcommand{\Out}{{\rm Out}}

\newcommand{\ul}[1]{\undertilde{#1}}
\newcommand{\ol}[1]{\overline{#1}}

\newcommand{\set}[1]{ \left\lbrace #1\right\rbrace}

\usepackage[backrefs]{amsrefs}

\newcommand{\FR}{{\rm Fr}_+(\phi,\psi)}
\newcommand{\cind}{ {\rm CIND}}
\newcommand{\type}{ {\rm type}}
\newcommand{\sub}{ {\rm Sub}}
\newcommand{\fix}{ {\rm Fix}}
\newcommand{\stab}{ {\rm stab}}
\newcommand{\malg}{ {\rm MALG}_\mu}
\newcommand{\core}{ {\rm core}}
\newcommand{\irs}{ {\rm IRS}}

\newcommand{\act}{ \text{A}}

\begin{document}
\title{\bf Co-induction and Invariant Random Subgroups}
\author{Alexander S. Kechris and Vibeke Quorning}
\maketitle

\begin{abstract}
In this paper we develop a co-induction operation which transforms an invariant random subgroup of a group into an invariant random subgroup of a larger group. 

We use this operation to construct new continuum size families of non-atomic, weakly mixing invariant random subgroups of certain classes of wreath products, HNN-extensions and free products with amalgamation. By use of small cancellation theory, we also construct a new continuum size family of non-atomic invariant random subgroups of $\F_2$ which are all invariant and weakly mixing with respect to the action of $\text{Aut}(\F_2)$. 

Moreover, for amenable groups $\Gamma\leq \Delta$, we obtain that the standard co-induction operation from the space of weak equivalence classes of $\Gamma$ to the space of weak equivalence classes of $\Delta$ is continuous if and only if $[\Delta :\Gamma]<\infty$ or $\core_\Delta(\Gamma)$ is trivial. For general groups we obtain that the co-induction operation is not continuous when $[\Delta:\Gamma]=\infty$. This answers a question raised by Burton and Kechris in \cite{BK18}. Independently such an answer was also obtained, using a different method, by Bernshteyn in \cite{B18}. 
\end{abstract}

\section{Introduction}

{\bf (A)} Let $\Gamma$ be a countable (discrete) group. We denote by ${\rm Sub}(\Gamma)$ the space of subgroups of $\Gamma$. It is a closed subset of the space $2^\Gamma$, so it is compact metrizable. The group $\Gamma$ acts continuously by conjugation on ${\rm Sub}(\Gamma)$ and a probability Borel measure on ${\rm Sub}(\Gamma)$ invariant under this action is called an {\bf invariant random subgroup (IRS)}. We denote by ${\rm IRS}(\Gamma)$ the space of invariant random subgroups for $\Gamma$, which, viewed as a subset of the space of probability Borel measures on ${\rm Sub}(\Gamma)$, with the usual weak$^*$-topology, is a compact metrizable space, in fact a Choquet simplex (see, e.g., \cite[page 95]{G03}). For example, if $N$ is a normal subgroup of $\Gamma$, then the Dirac measure $\delta_N$ is an IRS, and one can think of an IRS as a random version of the notion of normal subgroup.

The study of invariant random subgroups on various classes of groups has been an active area of research in the last several years, see, for example, \cite{BGK17} and the references contained therein, as well as \cite{Bo14}, \cite{TT-D14}, \cite{LM15}, \cite{Ge15}, \cite{BGN15}, \cite{O15}, \cite{LM15}, \cite{GL16}, \cite{HT16}, \cite{EG16}, \cite{BDLW16}, \cite{G17}, \cite{BBT17}, \cite{BT17}, \cite{DM17}, \cite{HY17}, \cite{Ge18}, \cite{BT18}, \cite{TT-D18}, \cite{BLT18}, \cite{GeL18}. One usually concentrates on the study of ergodic (with respect to the conjugacy action) invariant random subgroups, which are the extreme points of the Choquet simplex ${\rm IRS}(\Gamma)$. Among those are the atomic, ergodic invariant random subgroups, which are given by the uniform measure on the set of conjugates of a subgroup of $\Gamma$ that has finitely many conjugates. So we will naturally focus on the non-atomic, ergodic invariant random subgroups. There are groups for which there are continuum many such invariant random subgroups (e.g., the free non-abelian groups, see \cite{Bo15}) and others that have no such invariant random subgroups (e.g., lattices in simple higher rank Lie groups, see \cite{SZ94}). For more examples of both types, see the introduction of \cite{BGK17}. In the present paper, we will develop a method for constructing continuum many non-atomic, ergodic invariant random subgroups for certain classes of groups. In fact this method produces non-atomic, weakly mixing (for the conjugacy action) invariant random subgroups. We note here that weakly mixing is the strongest notion of mixing that a non-atomic IRS can have. This follows from the result of Tucker-Drob \cite{T-D15a}, which implies that any totally ergodic IRS must be atomic (recall here that a probability measure preserving action of a countable group is totally ergodic if every infinite subgroup acts ergodically; mixing or even mildly mixing actions are totally ergodic). 


\medskip
{\bf (B)} To explain our method, fix a standard (non-atomic) probability space $(X,\mu)$ and let $A(\Gamma, X, \mu)$ be the Polish space of measure preserving Borel actions of $\Gamma$ on $(X,\mu)$ with the usual weak topology (see, e.g., \cite{K10}). For each $a\in A(\Gamma, X, \mu), x\in X$, let ${\rm stab}_a (x)$ be the stabilizer of $x$. Then ${\stab}_a\colon X\to {\rm Sub} (\Gamma)$ is $\Gamma$-equivariant, so $\theta = ({\stab}_a)_*\mu\in {\rm IRS}(\Gamma)$. This IRS is called the {\bf type} of $a$, in symbols ${\rm type}(a)$. It was shown in \cite{AGV14} that every $\theta\in {\rm IRS}(\Gamma)$ is of the form ${\type}(a)$ for some action $a$.

If the group $\Gamma$ is contained in a countable group $\Delta$, $\Gamma\leq \Delta$, there is a canonical method of ``extending'' an action $a\in A(\Gamma, X, \mu)$ to an action $b\in A(\Delta, X^{\Delta/\Gamma}, \mu^{\Delta/\Gamma})$, where $\Delta/\Gamma$ is the set of left cosets of $\Gamma$ in $\Delta$. The action $b$ is called the {\bf co-induced} action of $a$ in symbols ${\rm CIND}^\Delta_\Gamma(a)$. For the basic properties of the co-induced action, see \cite{I11}, \cite{K10}. Our method for constructing invariant random subgroups is based on defining a notion of co-induction for invariant random subgroups. To formulate the precise statement, we use the following notation. If $F\subseteq \Gamma$ is finite, we let $N^\Gamma_F = \{\Lambda\in {\rm Sub}(\Gamma)\colon F\subseteq \Lambda\}$, be the ``positive'' basic open set given by $F$. It is not hard to see any $\theta \in {\rm IRS}(\Gamma)$ is completely determined by its values on these basic open sets. Also we let ${\rm core}_\Delta(\Gamma) = \bigcap_{\delta\in \Delta} \delta\Gamma\delta^{-1}$ be the {\bf normal core} of $\Gamma$ in $\Delta$. We now have:

\begin{thm}
Let $\Gamma\leq \Delta$ and $T\subseteq \Delta$ a transversal for the left cosets in $\Delta/\Gamma$. There exists a co-induction operation $\cind_\Gamma^\Delta\colon \irs(\Gamma)\rightarrow\irs(\Delta)$ such that for finite $F\subseteq \Delta$, 
\begin{equation*}
\cind_\Gamma^\Delta(\theta)(N_F^\Delta)=\begin{cases}
0 \qquad &\text{if} \quad F\nsubseteq\core_\Delta(\Gamma),\\
\prod_{t\in T}\theta(N_{t^{-1}Ft}^\Gamma) \qquad &\text{if} \quad F\subseteq \core_\Delta(\Gamma),
\end{cases}
\end{equation*}
 and
$\cind_\Gamma^\Delta(\type(a))=\type(\cind_\Gamma^\Delta(a))$ for all $a\in \text{A}(\Gamma,X,\mu)$. In particular, if
$\type(a)=\type(b)$, then $ \type(\cind_\Gamma^\Delta(a))=\type(\cind_\Gamma^\Delta(b))$. Also $\cind_\Gamma^\Delta(\theta)$ concentrates on ${\rm Sub}(\core_\Delta(\Gamma))$, for any $\theta \in {\rm IRS}(\Gamma)$.
\end{thm}
Although the expression above uses the transversal $T$, it is not hard to see that it is independent of the choice of $T$.


It is known, see \cite{I11}, that when the index of $\Gamma$ in $\Delta$ is infinite, then for any action $a$ the co-induced action $\cind_\Gamma^\Delta(a)$ is weakly mixing. Thus we have:
\begin{prop}
Let $\Gamma\leq \Delta$ with $[\Delta : \Gamma]=\infty$. Then $\cind_\Gamma^\Delta(\theta)\in \irs(\Delta)$ is weakly mixing, for any $\theta\in \irs(\Gamma)$.
\end{prop}

One can also characterize the non-atomicity of the co-induced IRS. Below for each $\theta\in {\rm IRS}(\Gamma)$, let the {\bf kernel} of $\theta$, in symbols ${\rm ker}(\theta)$, be the subgroup of $\Gamma$ given by ${\rm ker}(\theta) = \{\gamma \in \Gamma\colon \theta(N^\Gamma_{\gamma }) =1\}$, where $N^\Gamma_{\gamma } =   N^\Gamma_{\{\gamma\} }$.

\begin{prop}
Let $\Gamma\leq \Delta$ with $[\Delta :\Gamma]=\infty$ and $\theta\in \irs(\Gamma)$.

(i)  If $\cind_\Gamma^\Delta(\theta)$ is atomic, then $\cind_\Gamma^\Delta(\theta)=\delta_{\core_\Delta(\ker(\theta))}$.

(ii) Let $T\subseteq \Delta$ be a transversal for the left cosets in $\Delta/\Gamma$. Then $\cind_\Gamma^\Delta(\theta)$ is non-atomic if and only if there is $\gamma\in \core_\Delta(\Gamma)\setminus \core_\Delta(\ker(\theta))$ such that
\[\sum_{t\in T}\left(1-\theta(N_{t^{-1}\gamma t}^\Gamma)\right)<\infty\]
and $\theta(N_{t^{-1}\gamma t}^\Gamma)>0$ for all $t\in T$.

\end{prop}
We note here that one can also derive the following criterion for non-freeness of co-induced actions. Below for any action $a\in A(\Gamma, X, \mu)$ and $\gamma\in \Gamma$, we let ${\rm Fix}_a(\gamma) = \{ x\in X\colon \gamma^a(x) = x\}$.

\begin{prop} Let $\Gamma\leq \Delta$, $T\subseteq \Delta$ a transversal for the left cosets in $\Delta/\Gamma$ and $a\in \act(\Gamma,X,\mu)$. Then $\cind_\Gamma^\Delta(a)$ is not free if and only if for some $\gamma\in \core_\Delta(\Gamma)\setminus \set{e}$  we have
\[\sum_{t\in T}\left(1-\mu(\fix_a(t^{-1}\gamma t )) \right)<\infty\]
and $\mu\left(\fix_a(t^{-1}\gamma t)\right)>0$ for all $t\in T$.
\end{prop}

\medskip
{\bf (C)} We now apply co-induction to construct continuum many non-atomic, weakly mixing invariant random subgroups for several classes of groups. One approach makes use of the following criterion, where for a group $\Gamma$ and a subset $S\subseteq \Gamma$, we let $\langle S\rangle_\Gamma$  denote the subgroup generated by $S$ in $\Gamma
$ and  $\langle\langle S\rangle\rangle_\Gamma$ denote the {\bf normal closure} of $S$, i.e., the  smallest normal subgroup of $\Gamma$ containing $S$. 

\begin{prop}\label{criterion} Let $\Gamma\leq \Delta$ with $[\Delta :\Gamma]=\infty$. Suppose there exists a transversal $T=\set{t_i\mid i \in \N}$ for the left cosets in $\Delta/\Gamma$ and $\gamma_0\in \core_\Delta(\Gamma)$ such that the chain of normal subgroups $(\ol{\Gamma}_{k,T,\gamma_0})_{k\in \N}$, given by
\[\ol{\Gamma}_{k,T,\gamma_0}=\langle\langle t_i^{-1}\gamma_0t_i \mid i\geq k\rangle\rangle_\Gamma,\]
is not constant. Then $\Delta$ has continuum many non-atomic, weakly mixing invariant random subgroups.
\end{prop}
Applying this criterion, we construct new continuum size families of non-atomic, weakly mixing invariant random subgroups for the following classes of groups:

\medskip
(1)  All wreath products $H\wr G$, where $G,H$ are countable groups with $G$ infinite and $H$ non-trivial.

A different construction of such families is also contained in \cite{HY17}. Other results on invariant random subgroups of lamplighter groups are contained in \cite{BGK15}.

\medskip
(2) All HNN-extensions $G=\langle H,t \mid t^{-1}at=\phi(a), a\in A\rangle$, where $H$ is a countable group, $A\leq H$ and $\phi\colon A\rightarrow H$ an embedding such that 
\[\langle\langle A\cup \phi(A)\rangle\rangle \neq H,\]

In particular, this applies to the Baumslag-Solitar groups ${\rm BS}(n,m) = \langle x,t | \ tx^nt^{-1}=x^m\rangle$, where $m,n \in \Z\setminus\{0\}$ are not relatively prime.


\medskip
Applying in a different way the co-induction construction, we also find new continuum size families of non-atomic, weakly mixing invariant random subgroups for the following classes of groups:

\medskip
(3) All free products with amalgamation $ G\ast_A H$, where $G, H, A$ are countable groups satisfying that $A\unlhd G,H$ with $G/A$ non-trivial and $H/A$ infinite. 

Other such families have been constructed: (a) In \cite{BGK17} by using completely different techniques, including Pontryagin duality and a deep result of Adian \cite{A79} in combinatorial group theory; (b) In \cite{HY17}, for the free groups, using again different techniques, involving what they call {\it intersectional invariant random subgroups}. Our approach however is quite elementary.

In fact in  \cite{BGK17} it is shown that the free non-abelian groups admit continuum many non-atomic, weakly mixing invariant random subgroups that are moreover invariant under the full automorphism group (i.e., they are {\it characteristic random subgroups}). We also show in the last part of this paper how to use the criterion in \cref{criterion} to construct continuum many non-atomic, characteristic random subgroups for the free group $\F_2$ that are weakly mixing with respect to the full automorphism group. Our approach to that makes use of small cancellation theory. It is based on the following result that may be interesting in its own right. Below we view $\F_2$ as a normal subgroup of its automorphism group ${\rm Aut}(\F_2)$. Let $a,b$ be free generators of $\F_2$.

\begin{thm}\label{canc}
There is a transversal $T$ for the left cosets of $\F_2 = \langle a, b \rangle$ in ${\rm Aut}(\F_2)$ such that for $w= aba^2b^2 \cdots a^nb^n$, where $n>101$, we have that the set $\{\eta (w) \mid \eta \in T \}$ satisfies the $C'(1/6)$ cancellation property.
\end{thm}

In turn this has the following consequence concerning the natural action $a$ of the outer automorphism group $G =\Out(\F_2)= 
\Aut(\F_2)/\F_2$ on the set  of conjugacy classes $C$ of $\F_2$. It is well known that there is a conjugacy class $c\in C$ such that $\stab_a(c)=\set{e}$  (see \cite[page 45]{LS77}). Therefore for such $c\in C$ we have
\[c\cap \bigcup\{g\cdot^a c\mid g\in G\setminus\{e\}\} = \emptyset.\]
We obtain the following strengthening of this result.

\begin{cor} There exists a conjugacy class $c\in C$ such that $$c\cap \langle g\cdot^a c\mid g\in G\setminus\{e\}\rangle_{\F_2} = \emptyset,$$
that is, $c$ is disjoint from the (normal) subgroup generated by the conjugacy classes $g\cdot^a c$ for  $g\neq e$ .
\end{cor}
We note here that our proof of \cref{canc} makes use of the natural isomorphism of $\Out(\F_2)$ with ${\rm GL}_2(\Z)$ and we do not know if a similar result holds for $\F_n$, for $n> 2$.



\medskip
{\bf (D)} Actually the original motivation for the work in this paper came from a different problem. For $a,b\in A(\Gamma,X, \mu)$, let $a\preceq b$ be the pre-order of {\bf weak containment} and $a\simeq b \iff a\preceq b \ \& \ b\preceq a$ the notion of {\bf weak equivalence}; see \cite{K10} and \cite{BK18} for the theory of weak containment. Denote by $\undertilde{A}(\Gamma,X, \mu) = A(\Gamma, X, \mu)/\simeq$ the space of weak equivalence classes equipped with the compact, metrizable topology defined by Ab\'{e}rt-Elek, see \cite{AE11} and \cite{BK18}. For $a\in A(\Gamma, X, \mu)$, let $\undertilde{a}$ be its weak equivalence class. It turns out that $a\preceq b\implies {\rm CIND}_\Gamma^\Delta (a) \preceq {\rm CIND}_\Gamma^\Delta (b)$ and thus one has a well-defined function $\undertilde{\cind}_\Gamma^\Delta\colon \undertilde{A}(\Gamma,X,\mu)\rightarrow \undertilde{A}(\Delta,X^{\Delta/\Gamma},\mu^{\Delta/\Gamma})$, defined by $\undertilde{\cind}_\Gamma^\Delta (\undertilde{a}) = \undertilde{\cind_\Gamma^\Delta (a)}$.

The problem was raised in \cite{BK18} of whether the function $\undertilde{\cind}_\Gamma^\Delta$ is continuous. We can raise the same question concerning the co-induction operation on invariant random subgroups. We obtain here the following result:

\begin{prop}
Let $\Gamma\leq \Delta$. The map $\cind_\Gamma^\Delta\colon\irs(\Gamma)\rightarrow\irs(\Delta)$ is continuous if and only if either $[\Delta \colon \Gamma]<\infty$ or $\core_\Delta(\Gamma)$ is trivial.
\end{prop}
Since the space of weak equivalence classes is homeomorphic  to the space of invariant random subgroups (via the map $\ul{a} \mapsto {\rm type} (a)$) for any amenable group, see \cite{B16} and \cite{BK18}, we now have the following result:

\begin{thm} Let $\Gamma\leq \Delta$ and assume $\Delta$ amenable. Then we have that the map $\undertilde{\cind}_\Gamma^\Delta\colon \undertilde{A}(\Gamma,X,\mu)\rightarrow \undertilde{A}(\Delta,X^{\Delta/\Gamma},\mu^{\Delta/\Gamma})$ is continuous if and only if either $[\Delta :\Gamma]<\infty$ or ${\rm core}_\Delta(\Gamma)$ is trivial.
\end{thm}
One direction of this result is true {\it for any pair of groups}.
\begin{prop}
Let $\Gamma\leq \Delta$ and assume $[\Delta : \Gamma]=\infty$ and ${\rm core}_\Delta(\Gamma)$ is not trivial. Then $\undertilde{\cind}_\Gamma^\Delta \colon \undertilde{A}(\Gamma,X,\mu)\rightarrow \undertilde{A}(\Delta,X^{\Delta/\Gamma},\mu^{\Delta/\Gamma})$ is not continuous.
\end{prop}
In contrast, Bernshteyn in \cite{B18} recently showed that there are two (non-amenable) groups $\Gamma\leq \Delta$ with $[\Delta:\Gamma] =2$ such that $\undertilde{\cind}_\Gamma^\Delta \colon \undertilde{A}(\Gamma,X,\mu)\rightarrow \undertilde{A}(\Delta,X^{\Delta/\Gamma},\mu^{\Delta/\Gamma})$ is not continuous.

This paper is organized as follows. In Sections 2--4 we review concepts and results concerning the space of weak equivalence classes, co-induction of actions and invariant random subgroups. In Sections 5--7, we introduce and study the properties of co-induction on invariant random subgroups. Finally in Section 8, we use co-induction to construct continuum size families of non-atomic, weakly mixing invariant random subgroups for several classes of groups.

\medskip
{\it Acknowledgments.} ASK was partially supported by NSF Grant DMS-1464475. VQ was partially supported by Lars Hesselholt's Niels Bohr Professorship. We would like to thank Simon Thomas for a number of useful comments and in particular for bringing up the relevance of small cancellation theory to certain aspects of our work. We also would like to thank an anonymous referee for many useful remarks and corrections.

\section{The space of weak equivalence classes}
Here we will briefly introduce some of the basic notions that we will work with.

In this paper $(X,\mu)$ will always be a non-atomic standard probability space, that is, a Polish space equipped with its Borel $\sigma$-algebra and a non-atomic probability Borel measure. Recall that these are all isomorphic to $([0,1],\lambda)$, where $\lambda$ is the Lebesgue measure. The {\bf measure algebra} of $\mu$, denoted by $\malg$ is the algebra consisting of the Borel subsets of $X$ considered modulo $\mu$-null sets. This algebra can be equipped with a Polish topology induced by the complete metric
$d_\mu$ given by
\[d_\mu(A,B)=\mu(A\triangle B),\]
for all $A,B\in \malg$.

Let $\Aut(X,\mu)$ denote the group of measure preserving Borel isomorphisms of $(X,\mu)$, where we identify two isomorphisms if they agree almost everywhere. There are two natural topologies on $\Aut(X,\mu)$, which turn it into a topological group.  The {\bf weak topology}, $w$, on $\Aut(X,\mu)$ is the topology generated by the maps $\phi_A\colon \Aut(X,\mu)\rightarrow \malg$ given by $\phi_A(T)=T(A)$, where $A$ varies over all elements in $\malg$. A left invariant metric inducing this topology is given by
\[d_w(T,S)=\sum_{n\in \N}2^{-n-1}\mu(T(A_n)\triangle S(A_n)),\]
where $(A_n)_n\in \malg$ is a dense sequence. The {\bf uniform topology}, $u$,  on $\Aut(X,\mu)$ is defined by the two-sided complete invariant metric
\[d_u(S,T)=\mu\left(\set{x\in X\mid S(x)\neq T(x)}\right).\]
It is clear that the uniform topology is finer than the weak topology. Moreover, $(\Aut(X,\mu),w)$ is a Polish group, while the uniform topology is not separable. 

For a countable group $\Gamma$, we may represent each measure preserving action $\Gamma\ac^a(X,\mu)$ as a group homomorphism $h_a\colon \Gamma\rightarrow \Aut(X,\mu)$ given by $h_a(\gamma)(x)=\gamma\cdot^a x$. We will later on use the notation $\gamma^a$ instead of $h_a(\gamma)$. So we may consider the space of $\Gamma$-actions, $A(\Gamma,X,\mu)$, as a subset of $\Aut(X,\mu)^\Gamma$. In both the uniform and the weak topology $A(\Gamma,X,\mu)$ is closed in the product topology. Thus $A(\Gamma,X,\mu)$ is Polish in the topology inherited by the weak topology on $\Aut(X,\mu)$ and completely metrizable in the topology induced by the uniform topology on $\Aut(X,\mu)$. If nothing is specified, we will assume that $A(\Gamma,X,\mu)$ is equipped with the weak topology.

In \cite[Section 11]{K10} the notion of weak containment of actions is introduced. This notion is motivated by the analogous notion for unitary representations of groups and is defined as follows. Let $a,b\in A(\Gamma,X,\mu)$. We say that $a$ is {\bf weakly contained} in $b$ if for all $A_1,\ldots, A_n\in \malg$, $F\subseteq \Gamma$ finite and $\epsilon>0$ there is $B_1,\ldots, B_n\in \malg$ such that
\[|\mu(\gamma^a(A_i)\cap A_j)- \mu(\gamma^b(B_i)\cap B_j)|<\epsilon,\]
for all $i,j\leq n$ and $\gamma\in F$. If $a$ is weakly contained in $b$ we write $a\preceq b$. If $b$ is weakly contained in $a$, as well, we say that $a$ and $b$ are {\bf weakly equivalent} and write $a\simeq b$. Another way to characterize weak containment is as follows. Fix an enumeration $\Gamma=\set{\gamma_i\mid i\in \N}$ and for each $k > 1$, let $\p_k$ denote the set of all Borel partitions of $X$ into $k$ pieces. For each $a\in A(\Gamma,X,\mu)$, $n,k >1$ and  $P=(A_0,\ldots, A_{k-1})\in \p_k$ we let $M_{n,k}^{P}(a)\in [0,1]^{n\times k\times k}$ be given by
\[M_{n,k}^{P}(a)(m,i,j)=\mu(\gamma_m^a(A_i)\cap A_j)\]
for $m<n$ and $i,j<k$. Put 
\[C_{n,k}(a)=\ol{\set{M_{n,k}^{P}(a)\mid P\in \p_k}},\]
that is, the closure of the set $\set{M_{n,k}^P(a)\mid P\in \p_k}$ in $[0,1]^{n\times k\times k}$. Then it is straightforward to check that we have 
\begin{align*}
a\preceq b &\iff (\forall n,k >1) (C_{n,k}(a)\subseteq C_{n,k}(b))\\
a\simeq b &\iff (\forall n,k >1) (C_{n,k}(a)=C_{n,k}(b)),
\end{align*}
for all $a,b\in A(\Gamma,X,\mu)$.

Consider the set of weak equivalence classes $\ul{A}(\Gamma,X,\mu)=A(\Gamma,X,\mu)/\simeq$. For each $a\in A(\Gamma,X,\mu)$, we let $\ul{a}\in \ul{A}(\Gamma,X,\mu)$ denote its weak equivalence class. By the above, the map $\iota\colon \ul{A}(\Gamma,X,\mu)\rightarrow \prod_{n,k>1}F([0,1]^{n\times k \times k})$ given by $$\iota(\ul{a})=(C_{n,k}(a))_{n,k>1}$$ is an injection. Here $F([0,1]^{n\times k \times k})$ denotes the space of all closed subsets $[0,1]^{n\times k\times k}$. It follows by \cite[Theorem 4.26]{K95}  that $F([0,1]^{n\times k \times k})$ is a compact metrizable space, when equipped with the Vietoris topology. Fix a complete metric $d$ on $[0,1]^{n\times k \times k}$ with $\text{diam}_d([0,1]^{n\times k \times k})=1$. If for all $K,L\in F([0,1]^{n\times k\times k})$, we let 
\[\delta_{n,k}(K,L)=\max_{x\in K}\inf_{y\in L}d(x,y),\]
 then
\begin{equation*}
d_{n,k}(K,L)=\begin{cases}
0 \quad &\text{if} \ \  L=K=\emptyset\\
1 \quad &\text{if} \ \  (L=\emptyset \vee K=\emptyset)\wedge K\neq L\\
\max\set{ \delta_{n,k}(K,L), \delta_{n,k}(L,K)} \ \ &\text{if} \quad K,L\neq \emptyset ,
\end{cases}
\end{equation*}
is a complete metric on $F([0,1]^{n\times k \times k})$.
It is proven in \cite[Theorem 4]{AE11} 
 that the image $\iota(\ul{A}(\Gamma,X,\mu))$ is closed in $\prod_{n,k>1}F([0,1]^{n\times k \times k})$. Thus, by transferring back the subspace topology, we obtain a compact metrizable topology on $\ul{A}(\Gamma,X,\mu)$. Moreover, we obtain that
\[d(\ul{a},\ul{b})=\sum_{n,k >1}2^{-n-k}d_{n,k}(C_{n,k}(a),C_{n,k}(b))\]
is a metric inducing the topology on $\ul{A}(\Gamma,X,\mu)$. We will from now on assume that $\ul{A}(\Gamma, X, \mu)$ is equipped with this topology.
\section{Co-induction of actions}
In this section we will discuss the co-induction operation for actions, which transforms an action of a subgroup to an action of the bigger group.

Let $\Gamma\leq \Delta$ and fix a transversal $T\subseteq \Delta$ for the left cosets in $\Delta/\Gamma$. We then have an action $\sigma_T\colon \Delta\times T\rightarrow T$ given by
\[\sigma_T(\delta,t) = \tilde{t} \iff \tilde{t}\Gamma = \delta t\Gamma\]
and a cocycle $\rho_T\colon \Delta \times T\rightarrow \Gamma$ for this action given by
\[\rho_T(\delta,t)=\sigma_T(\delta,t)^{-1}\delta  t.\]
Now for each $a\in A(\Gamma,X,\mu)$ we obtain the co-induced action $a_T\in A(\Delta,X^T,\mu^T)$ by
\[(\delta \cdot^{a_T} f)(t)=\rho_T(\delta^{-1},t)^{-1}\cdot^a f(\sigma_T(\delta^{-1},t)).\]
By considering the natural bijection $\iota_T\colon \Delta/\Gamma \rightarrow T$ we may view the co-induced action $a_T\in A(\Delta,X^{\Delta/\Gamma},\mu^{\Delta/\Gamma})$ by  letting
\[(\delta \cdot^{a_T} f)(\delta_0\Gamma)=\rho_T(\delta^{-1},\iota_T(\delta_0\Gamma))^{-1}\cdot^a f(\delta^{-1}\delta_0\Gamma).\]

\begin{prop}\label{coindeptrans} Let $\Gamma\leq \Delta$ be groups. If $T,S\subseteq \Delta$  are transversals for the left cosets in $\Delta/\Gamma$  and $a\in \act(\Gamma,X,\mu)$, then $a_T,a_S\in \act(\Delta,X^{\Delta/\Gamma},\mu^{\Delta/\Gamma})$ are isomorphic.
\end{prop}

\begin{proof}
First, consider the map $\iota\colon \Delta/\Gamma\rightarrow \Gamma$ given by
$\iota(\delta_0\Gamma)=\iota_S(\delta_0\Gamma)^{-1}\iota_T(\delta_0\Gamma)$. Note that
\begin{align*}
\rho_S(\delta,\iota_S(\delta_0\Gamma))=\iota(\delta\delta_0\Gamma)\rho_T(\delta, \iota_T(\delta_0\Gamma))\iota(\delta_0\Gamma)^{-1}
\end{align*}
and thus $\phi\colon X^{\Delta/\Gamma}\rightarrow X^{\Delta/\Gamma}$ given by $\phi(f)(\delta_0\Gamma)=\iota(\delta_0\Gamma)\cdot^a f(\delta_0\Gamma)$ satisfies
\[\phi(\delta \cdot^{a_T} f)(\delta_0\Gamma)=\delta \cdot^{a_S} \phi(f)(\delta_0\Gamma).  \]
\end{proof}

\cref{coindeptrans} ensures that we may omit the transversal in the notation above and for an action $a\in A(\Gamma,X,\mu)$ let $\cind_\Gamma^\Delta(a)\in A(\Delta,X^{\Delta/\Gamma},\mu^{\Delta/\Gamma})$ denote the co-induced action with respect to some transversal. In \cite[Section 10(G)]{K10} it is proven that the map $\cind_\Gamma^\Delta\colon A(\Gamma,X,\mu)\rightarrow A(\Delta,X^{\Delta/\Gamma},\mu^{\Delta/\Gamma})$ is continuous. Moreover, it is easily checked that $a$ is a factor of $\cind_\Gamma^\Delta(a)_{|\Gamma}$ via the map $f\mapsto f(\Gamma)$ from $X^{\Delta/\Gamma}$ to $X$. The operation also satisfies the following ``chain'' rule. 

\begin{prop} Let $\Lambda\leq \Gamma \leq \Delta$ and $a\in \act(\Lambda,X,\mu)$. Then the actions ${\rm CIND}_\Gamma^\Delta\left({\rm CIND}_\Lambda^\Gamma(a)\right)$
 and ${\rm CIND}_\Lambda^\Delta(a)$
 are isomorphic.
\end{prop}

\begin{proof} Let $T\subseteq \Delta$ and $S\subseteq \Gamma$ be transversals for the left cosets in $\Delta/\Gamma$ and $\Gamma/\Lambda$, respectively. Then it is easily seen that $TS=\set{ts\mid t\in T, s\in S}$ is a transversal for $\Delta/\Lambda$. One may check that
\[\sigma_{TS}(\delta,ts)=\sigma_T(\delta,t)\sigma_S(\rho_T(\delta,t),s)\]
and hence
\[\rho_{TS}(\delta,ts)=\rho_S(\rho_T(\delta,t),s)\]
for all $s\in S, t\in T$ and $\delta\in \Delta$.

Next, consider the map $\phi\colon (X^S)^T\rightarrow X^{TS}$ given by $\phi(f)(ts)=(f(t))(s)$. We have that
\begin{align*}
\phi(\delta \cdot^{(a_S)_T} f)(ts)&=\left((\delta \cdot^{(a_S)_T} f)(t)\right)(s)\\
&= \left(\rho_T(\delta^{-1},t)^{-1}\cdot^{a_S} f(\sigma_T(\delta^{-1},t))\right)(s)\\
&=\rho_S(\rho_T(\delta^{-1},t),s)^{-1} \cdot^a \left(f(\sigma_T(\delta^{-1},t))\right)(\sigma_S(\rho_T(\delta^{-1},t),s))\\
&= \rho_{TS}(\delta^{-1},ts)^{-1}\cdot^a \phi(f)(\sigma_{TS}(\delta^{-1},ts))\\
&= \left(\delta \cdot^{a_{TS}} \phi(f)\right)(ts).
\end{align*}
Thus the actions are isomorphic, as wanted.
\end{proof}

In \cite[Lemma 2.2.]{I11} different mixing properties of the co-induced actions are studied. Among other things  the following is proved.
\begin{prop}\label{mixact} Let $\Gamma\leq \Delta$ and $a\in A(\Gamma,X,\mu)$. 
\begin{itemize}
\item[(1)] If $[\Delta :\Gamma]=\infty$, then ${\rm CIND}_\Gamma^\Delta(a)$ is weakly mixing.
\item[(2)] If $[\Delta:\Gamma]<\infty$, then ${\rm CIND}_\Gamma^\Delta(a)$ is weakly mixing if and only if $a$ is weakly mixing.
\item[(3)] ${\rm CIND}_\Gamma^\Delta(a)_{|\Gamma}$ is weakly mixing if and only if $a$ is weakly mixing.
\end{itemize}
\end{prop}

We end this section with the connection to the space of weak equivalence classes.
The following result is proven in \cite[Proposition A.1]{K12} and ensures that the co-induction operation is invariant under weak equivalence.
\begin{prop} Let $\Gamma\leq \Delta$ and $a,b\in A(\Gamma,X,\mu)$. If $a\preceq b$, then we have  ${\rm CIND}_\Gamma^\Delta(a)\preceq {\rm CIND}_\Gamma^\Delta(b)$.
\end{prop}

Thus we have that co-induction descends to a well defined operation on weak equivalence classes. So, if $\Gamma\leq \Delta$ we may for each $\ul{a}\in \ul{A}(\Gamma,X,\mu)$ assign $\ul{\cind}_\Gamma^\Delta(\ul{a})=\ul{\cind_\Gamma^\Delta(a)}\in \ul{A}(\Delta,X^{\Delta/\Gamma},\mu^{\Delta/\Gamma})$. In \cref{cont} we will address the question of continuity of the operation $\ul{\cind}_\Gamma^\Delta\colon  \ul{A}(\Gamma,X,\mu)\rightarrow \ul{A}(\Delta,X^{\Delta/\Gamma},\mu^{\Delta/\Gamma})$.

\section{Invariant random subgroups}
In this section we will discuss the notion of invariant random subgroups, which can be seen as a random version of normal subgroups.

Fix a countable group $\Gamma$ and let $\sub(\Gamma)\subseteq 2^\Gamma$ denote the set of all subgroups of $\Gamma$. Note that this is a closed subset and thus $\sub(\Gamma)$ is a compact Polish space. Moreover, we have a natural continuous action $\Gamma\ac \sub(\Gamma)$ given by $\gamma\cdot \Lambda = \gamma\Lambda\gamma^{-1}$. An {\bf invariant random subgroup} of $\Gamma$ is a $\Gamma$-invariant  probability Borel measure on $\sub(\Gamma)$. We denote by $\irs(\Gamma)$ the set of all invariant random subgroups and we will use the abbreviation IRS for ``invariant random subgroup''.

It is easily seen that for any normal subgroup $\Lambda\unlhd \Gamma$, the Dirac measure $\delta_\Lambda$ is an IRS. For a more interesting example, let $a\in A(\Gamma,X,\mu)$ and consider the map $\stab_a \colon X\rightarrow \sub(\Gamma) $, which assigns to each point $x\in X$ its stabilizer subgroup $\stab_a(x)$. It is straightforward to verify that this map is Borel and $\Gamma$-equivariant. Therefore we have that the pushforward of $\mu$ via this map, $(\stab_a)_*\mu$, is an IRS of $\Gamma$. The IRS $(\stab_a)_*\mu$ is denoted by $\type(a)$. Actually it is proven in \cite[Proposition 13]{AGV14} that 
any IRS on $\Gamma$ arises in this manner. Thus the study of measure preserving actions and invariant random subgroups are closely related.

The first part of the result below is proved in \cite[Section 4]{AE11} (see also \cite[Theorem 5.2]{T-D15b}) and states that the map $\type$ is invariant under weak equivalence. A proof of the second part is found in \cite[Proposition 5.1]{B16} (see also \cite[Theorem 1.8]{T-D15b} and \cite[Theorem 9]{E12}).
\begin{thm} Let $\Gamma$ be a countable group and $a,b\in A(\Gamma,X,\mu)$. 
\begin{itemize}
\item[(1)] If $a\simeq b$, then $\type(a)=\type(b)$.
\item[(2)] If $\Gamma$ is amenable and $\type(a)=\type(b)$, then $a\simeq b$.
\end{itemize} 
\end{thm}

So from the above we have a well-defined surjective map $$\type\colon \ul{A}(\Gamma,X,\mu)\rightarrow \irs(\Gamma)$$ given by $\type(\ul{a})=\type(a)$. Moreover, whenever $\Gamma$ is amenable this map is a bijection. This clearly fails in the case of non-amenable groups, as these have several weakly inequivalent free actions. 

We have that $\irs(\Gamma)$ is a closed subset of the compact Polish space $P(\sub(\Gamma))$ of all probability Borel measure on $\sub(\Gamma)$.  Thus $\irs(\Gamma)$ is also a compact Polish space in the subspace topology. For each finite $F\subseteq\Gamma$ let
\[N_F^\Gamma =\set{\Lambda\in \sub(\Gamma)\mid F\subseteq \Gamma}.\]
Then the sets $N_F^\Gamma$, for  $F\subseteq \Gamma$ finite, constitute a family of clopen subsets which is closed under finite intersections and generates the Borel structure of $\sub(\Gamma)$. Thus it follows from the $\pi$-$\lambda$ Theorem (see \cite[Theorem 10.1]{K95})
that if $\theta,\eta\in \irs(\Gamma)$ agree on these sets, then $\theta=\eta$. Using this together with the compactness of the space we obtain the following lemma.
\begin{lem}\label{convbasic} Let $\Gamma$ be a countable group and $(\theta_n)_n,\theta\in \irs(\Gamma)$. Then $\theta_n\rightarrow \theta$ as $n\rightarrow \infty$ if and only if 
$\theta_n(N_F^\Gamma)\rightarrow \theta(N_F^\Gamma)$ as $n\rightarrow \infty$,
for all finite $F\subseteq \Gamma$.
\end{lem}
\begin{proof}
The left to right implication follows directly from The Portmanteau Theorem (see \cite[Theorem 17.20]{K95}). For the other implication, assume $\theta_n \nrightarrow \theta$ as $n\rightarrow \infty$. By compactness there is a subsequence $(\theta_{n_i})_i$ and $\eta\in \irs(\Gamma)$ such that $\eta\neq \theta$ and $\theta_{n_i}\rightarrow \eta$. Since $\eta\neq \theta$ there exists finite $F\subseteq \Gamma$ such that $\eta(N_F^\Gamma)\neq \theta(N_F^\Gamma)$. Thus $\theta_{n_i}(N_F^\Gamma)\nrightarrow \theta(N_F^\Gamma)$ when $n\rightarrow \infty$, as wanted.
\end{proof}
Thus in order to study topological properties of $\irs(\Gamma)$ it  suffices to consider these basic sets. For $\gamma\in \Gamma$, we write $N_\gamma^\Gamma$ instead of $N_{\set{\gamma}}^\Gamma$. Note that in this case we have $\type(a)\left(N_\gamma^\Gamma\right)=\mu\left(\fix_a(\gamma)\right)$ for all $a\in A(\Gamma,X,\mu)$, where $\fix_a(\gamma)=\set{x\in X\mid \gamma \cdot^a x=x}$.

The following theorem is proved in \cite[Theorem 5.2]{T-D15b}.
\begin{thm} Let $\Gamma$ be a countable group. The map $\type\colon \ul{A}(\Gamma,X,\mu)\rightarrow \irs(\Gamma)$ is continuous.
\end{thm} 
\noindent
In particular it follows that, when $\Gamma$ is amenable, $\type\colon \ul{A}(\Gamma,X,\mu)\rightarrow \irs(\Gamma)$ is a homeomorphism. Moreover, if we let $$\ul{\text{FR}}(\Gamma,X,\mu)=\type^{-1}(\delta_{\set{e}}),$$
we obtain that $\ul{\text{FR}}(\Gamma,X,\mu)$ is a closed subspace of $\ul{A}(\Gamma,X,\mu)$. Note that $\ul{\text{FR}}(\Gamma,X,\mu)$ is the space consisting of all the weak equivalence classes of the free actions of $\Gamma$.

\section{Co-induction of invariant random subgroups}
In this section we will use the close connection between actions and invariant random subgroups to construct a co-induction operation on the invariant random subgroups.

In the following, whenever we have two groups $\Gamma\leq \Delta$, we let $$\core_\Delta(\Gamma)=\bigcap_{\delta\in \Delta}\delta\Gamma\delta^{-1}.$$  
The subgroup $\core_\Delta(\Gamma)$ is called the {\bf normal core} of $\Gamma$ in $\Delta$. It is straightforward to prove that $\core_\Delta(\Gamma)=\cap_{t\in T}t\Gamma t^{-1}$ for any transversal $T\subseteq \Delta$ of the left cosets $\Delta/\Gamma$.

\begin{thm}\label{coinducedIRS} Let $\Gamma\leq \Delta$ and $T\subseteq \Delta$ a transversal for the left cosets in $\Delta/\Gamma$. There exists a co-induction operation $\cind_\Gamma^\Delta\colon \irs(\Gamma)\rightarrow\irs(\Delta)$ such that 
\begin{equation*}
\cind_\Gamma^\Delta(\theta)(N_F^\Delta)=\begin{cases}
0 \qquad &\text{if} \quad F\nsubseteq\core_\Delta(\Gamma),\\
\prod_{t\in T}\theta(N_{t^{-1}Ft}^\Gamma) \qquad &\text{if} \quad F\subseteq \core_\Delta(\Gamma),
\end{cases}
\end{equation*}
 and
$\cind_\Gamma^\Delta(\type(a))=\type(\cind_\Gamma^\Delta(a))$, for all $a\in \text{A}(\Gamma,X,\mu)$.
\end{thm}

\begin{proof}
Let $\theta\in \text{IRS}(\Gamma)$ and fix $a\in A(\Gamma, X,\mu)$ such that $\type(a)=\theta$. We will then show that $\type(\cind_\Gamma^\Delta(a))\in \text{IRS}(\Delta)$ satisfies
\begin{equation*}
\type(\cind_\Gamma^\Delta(a))(N_F^\Delta)=\begin{cases}
0 \qquad &\text{if} \quad F\nsubseteq\text{core}_\Delta(\Gamma),\\
\prod_{t\in T}\theta(N_{t^{-1}Ft}^\Gamma) \qquad &\text{if} \quad F\subseteq \text{core}_\Delta(\Gamma).
\end{cases}.
\end{equation*}
In particular this shows that co-induction does not depend on the choice of $a$.

First we will prove the following claim.\\

\textbf{Claim:} {\it We have that  $\type(\cind_\Gamma^\Delta(a))({ \rm Sub}\left({ \rm core}_\Delta
(\Gamma))\right)=1$ and 	\[\type(\cind_\Gamma^\Delta(a))_{|{ \rm Sub}({ \rm core}_\Delta(\Gamma))}=\type(\cind_\Gamma^\Delta(a)_{|{ \rm core}_\Delta(\Gamma)}).\]}
{\it Proof of Claim.}
Recall that $\sigma\colon \Delta\times T\rightarrow T$ is given by $$\sigma(\delta,t)=\tilde{t} \iff \delta t\Gamma =\tilde{t}\Gamma$$
and $\rho \colon \Delta\times T\rightarrow \Gamma$ by $\rho(\delta, t)= \sigma(\delta,t)^{-1}\delta t$. Note that $\sigma(\delta , \cdot )\colon T\rightarrow T$ is the identity if and only if $\delta\in \text{core}_\Delta(\Gamma)$. Recall also that for $a\in A(\Gamma,X,\mu)$ we have $\beta_a=\cind_\Gamma^\Delta(a)$ is given by
\[(\delta\cdot^{\beta_a} f)(t)=\rho(\delta^{-1},t)^{-1}\cdot^a f(\sigma(\delta^{-1},t))\]
for all $f\in X^T$. 
 Now let $\delta \in \Delta\setminus \core_\Delta(\Gamma)$ and fix $t\in T$ such that $\sigma(\delta^{-1},t)\neq t$. Then we have
\begin{align*}
\mu^T(\fix_{\beta_a}(\delta))&=\mu^T(\set{f\in X^T\mid \delta \cdot^{\beta_a} f =f})\\
&\leq \mu^2(\set{g\in X^2 \mid \rho(\delta^{-1},t)^{-1} \cdot^a g(0)=g(1)})\\
&=0.
\end{align*}
So, the set
\begin{align*}
X^T\setminus \stab_{\beta_a}^{-1}(\sub(\text{core}_\Delta(\Gamma))) = \bigcup_{\delta\in \Delta\setminus \text{core}_\Delta(\Gamma)}\fix_{\beta_a}(\delta)
\end{align*}
is a null set and thus $\type(\cind_\Gamma^\Delta(a))(\sub(\text{core}_\Delta(\Gamma)))=1$. Moreover,
\[\stab_{\beta_a}(f)\in \sub(\text{core}_\Delta(\Gamma)) \iff \stab_{\beta_a}(f)=\stab_{{\beta_a|\text{core}_\Delta(\Gamma)}}(f),\]
hence $\type(\cind_\Gamma^\Delta(a))_{|\sub(\text{core}_\Delta(\Gamma))}=\type(\cind_\Gamma^\Delta(a)_{|\text{core}_\Delta(\Gamma)})$.
\begin{flushright}
$\diamond$
\end{flushright}

We now have\begin{equation*}
\type(\cind_\Gamma^\Delta(a))(N_F^\Delta)=\begin{cases}
0 \qquad &\text{if} \quad F\nsubseteq\text{core}_\Delta(\Gamma),\\
\type(\prod_{t\in T}a_t)(N_F^{\text{core}_\Delta(\Gamma)}) \qquad &\text{if} \quad F\subseteq \text{core}_\Delta(\Gamma),
\end{cases}
\end{equation*}
where $a_t\in A(\text{core}_\Delta(\Gamma),X,\mu)$ is given by $\gamma \cdot^{a_t} x = t^{-1}\gamma t \cdot^a x$ for each $t\in T$. Thus, as
\[\type(\prod_{t\in T}a_t)(N_F^{\text{core}_\Delta(\Gamma)})=\prod_{t\in T}\type(a)(N_{t^{-1}Ft}^\Gamma)=\prod_{t\in T}\theta(N_{t^{-1}Ft}^\Gamma),\]
the conclusion follows.
\end{proof}

Note that it follows by the $\Gamma$-invariance of $\theta$, that if $T,\tilde{T}\subseteq \Delta$ are both transversals for the left cosets $\Delta/\Gamma$, then
\[\prod_{t\in T}\theta(N_{t^{-1}Ft})=\prod_{\tilde{t}\in \tilde{T}}\theta(N_{\tilde{t}^{-1}F\tilde{t}}).\]
Thus $\cind_\Gamma^\Delta(\theta)$ does not depend on the chosen transversal.

\begin{remark}\label{alternative cind} (1) There is another way to describe the co-induced IRS as a factor of a certain action of $\Delta$ via an infinite intersection operation. Indeed let $\theta\in \irs(\Gamma)$ and view $\theta$ as a probability Borel measure on $\sub(\Delta)$. Let $T$ be a transversal for the left cosets of $\Delta/\Gamma$. For each $t\in T$ define the probability Borel measure, $\theta_t$, on $\sub(\Delta)$ 
to be the the push-forward of $\theta$ through the map $\Lambda\mapsto t\Lambda t^{-1}$ from $\sub(\Delta)$ to $\sub(\Delta)$. Thus 
\[\theta_t(N_F^\Delta)=\theta(N_{t^{-1}Ft}^\Delta),\]
for all finite $F\subseteq \Delta$. 


 Then 
\[
\theta_\infty = \prod_{t\in T}\theta_t
\]
is a probability Borel measure on $\sub(\Delta)^T$. Moreover, we have an action $\Delta\ac^a \sub(\Delta)^T$ given by $$\delta \cdot^a (\Lambda_t)_{t\in T}=(\delta \Lambda_{\delta^{-1}\cdot t}\delta^{-1})_{t\in T}.$$
Note that $\theta_\infty$ is $a$-invariant and that  $I\colon \sub(\Delta)^{T}\rightarrow \sub(\Delta)$ given by
\[I((\Lambda_t)_{t\in T})=\bigcap_{t\in T}\Lambda_t\]
is a Borel map. In fact we have 
\[I\left(\delta\cdot^a(\Lambda_t)_{t\in T}\right)=\bigcap_{t\in T }\delta \Lambda_t \delta^{-1} = \delta I\left((\Lambda_t)_{t \in T}\right)\delta^{-1}.\]
Thus if we let $\theta^*$ denote the push-forward of $\theta_\infty$ through $I$, we obtain that $\theta^*\in \irs(\Delta)$. To see that $\theta^*=\cind_\Gamma^\Delta(\theta)$ note that
\begin{equation*}
\theta^*(N_F^\Delta)=\prod_{t\in T}\theta_t(N_F^\Delta)=\begin{cases}
0 \qquad &\text{if}\quad F\nsubseteq \core_\Delta(\Gamma),\\
\prod_{t\in T}\theta(N_{t^{-1}Ft}^\Gamma) \qquad &\text{if} \quad F\subseteq \core_\Delta(\Gamma).
\end{cases}
\end{equation*}

(2) Let now $b\in A(\Gamma, X, \mu)$ be such that ${\rm type} (b) = \theta$ and let $c= {\rm CIND}_\Gamma^\Delta(b) \in A(\Delta, X^T, \mu^T)$. Define $F_b\colon X^T\to {\rm Sub}(\Delta)^T$ by $$(x_t)_{t\in T}\mapsto (t \ {\rm stab}_b (x_t) \  t^{-1})_{t\in T}.$$ Then $F_b$ is $\Delta$-equivariant, where $\Delta$ acts on $X^T$ by $c$ and on ${\rm Sub}(\Delta)^T$ by $a$, and $(F_a)_* (\mu ^T) = \theta_\infty$, i.e., $a$ is a factor of $c$, therefore $a$ is weakly mixing if $[\Delta: \Gamma]=\infty$.

(3) One can also define $J\colon {\rm Sub}(\Delta)^T \to {\rm Sub}(\Delta)$ by 
\[
J((\Lambda_t)_{t\in T})=\langle\Lambda_t\colon t\in T\rangle,
\]
the subgroup generated by $\{ \Lambda_t\}_{t\in T}$. Then $J$ is also $\Delta$-equivariant and thus if $\theta^{**}$ is the push-forward of $\theta_\infty$ by $J$, then $\theta^{**}\in {\rm IRS}(\Delta)$ and it is weakly mixing if $[\Delta\colon\Gamma] = \infty$. We do not know however when it is non-atomic.

(4) Yet another way to describe the co-induced IRS is as follows. Let $\tilde{I}\colon \sub(\Delta)^{T}\rightarrow \sub(\Delta)$ be given by
\[\tilde{I}\left((\Lambda_t)_{t\in T}\right)=\bigcap_{t\in T}t\Lambda_t t^{-1}.\]
Then for any $\theta\in\irs(\Gamma)$ we have that $\cind_\Gamma^\Delta(\theta)$ is the push-forward of $\theta^T$ (viewed as a probability Borel measure on $\sub(\Delta)^T$) through $\tilde{I}$. However, in this case it is not clear that $\tilde{I}$ is $\Delta$-equivariant.

\end{remark}

An easy consequence of  \cref{coinducedIRS} above is that the type of the co-induced action only depends on the type of the action.

\begin{cor} Let $\Gamma\leq \Delta$ and $a,b\in \text{A}(\Gamma,X,\mu)$. If
$\type(a)=\type(b)$ then $ \type(\cind_\Gamma^\Delta(a))=\type(\cind_\Gamma^\Delta(b))$.
\end{cor}
\noindent
\begin{remark} If $\Gamma\leq \Delta$ and $\theta\in \irs(\Gamma)$, we may also view $\cind_\Gamma^\Delta(\theta)$ as an element of $\irs(\Gamma)$. In \cite{BGK17} the notion of a {\bf characteristic random subgroup} is defined to be an IRS which moreover is invariant under the action of the full automorphism group. So, as any group $\Gamma$ is contained in a countable group $\Delta$ such that for densely many $\phi\in \Aut(\Gamma)$, there is $\delta\in \Delta$ such that $\phi(\gamma)=\delta\gamma\delta^{-1}$ for all $\gamma\in \Gamma$, we can use the co-induction operation $\cind_\Gamma^\Delta$ to transform invariant random subgroups on $\Gamma$ into characteristic random subgroups on $\Gamma$. We will do this for the free group of rank two in Section 8.E.



\end{remark}

\section{Continuity of co-induction}\label{cont}
We will consider here continuity properties of the co-induction operation. First, on the level of  invariant random subgroups we have the following result.

\begin{prop}\label{contirs} Let $\Gamma\leq \Delta$. The map $\cind_\Gamma^\Delta\colon\irs(\Gamma)\rightarrow\irs(\Delta)$ is continuous if and only if either $[\Delta \colon \Gamma]<\infty$ or $\core_\Delta(\Gamma)=\set{e}$.
\end{prop}
\begin{proof}
It is easily seen that if $\core_\Delta(\Gamma)=\set{e}$, then $\cind_\Gamma^\Delta(\theta)=\delta_{\set{e}}$  for any $\theta\in \irs(\Gamma)$. Thus in this case the co-induction operation is constant. If $[\Delta :\Gamma]<\infty$, then  the operation is continuous because the product in the definition is finite.

Conversely, assume $[\Delta :\Gamma]=\infty$ and $\text{core}_\Delta(\Gamma)\neq \set{e}$.
For each $n\in \N$ let $\theta_n=2^{-n}\delta_{\set{e}}+(1-2^{-n})\delta_\Gamma$.
Then $\theta_n\rightarrow \theta =\delta_{\Gamma}$ as $n\rightarrow \infty$ in $\irs(\Gamma)$. But we have
$$\cind_\Gamma^\Delta(\theta_n)(N_F^\Delta)=\prod_{t\in T}\theta_n(N_{t^{-1}Ft}^\Gamma)=0$$
for any $\set{e}\subsetneq F\subseteq \core_\Delta(\Gamma)$ finite, while
$$\cind_\Gamma^\Delta(\theta)(N_F^\Delta)=\prod_{t\in T}\delta_{\Gamma}(N_{t^{-1}Ft}^\Gamma)=1$$
for all $F\subseteq \core_\Delta(\Gamma)$ finite.
\end{proof}

Since $\type$ is a homeomorphism between the space of weak equivalence classes of an amenable group and the space of invariant random subgroups on the group, we now have the following corollary.

\begin{cor} Let $\Gamma\leq \Delta$ and assume $\Delta$ amenable. Then we have that the map $\undertilde{\cind}_\Gamma^\Delta\colon \undertilde{A}(\Gamma,X,\mu)\rightarrow \undertilde{A}(\Delta,X^{\Delta/\Gamma},\mu^{\Delta/\Gamma})$ is continuous if and only if either $[\Delta :\Gamma]<\infty$ or ${\rm core}_\Delta(\Gamma)= \set{e}$.
\end{cor}
One implication holds in general. To prove this we will use a sequence of weak equivalence classes, which converges to the weak equivalence class of the trivial action in $\ul{A}(\Gamma,X,\mu)$. We denote the trivial action by $i_\Gamma$.
\begin{prop}\label{contcind} Let $\Gamma\leq \Delta$ and assume $[\Delta : \Gamma]=\infty$ and ${\rm core}_\Delta(\Gamma)\neq \set{e}$. Then $\undertilde{\cind}_\Gamma^\Delta \colon \undertilde{A}(\Gamma,X,\mu)\rightarrow \undertilde{A}(\Delta,X,\mu)$ is not continuous.
\end{prop}
\begin{proof}
Consider a sequence of actions $(a_n)_n\in A(\Gamma,X,\mu)$ for which there exists $(B_n)_n\subseteq X$ Borel satisfying that $\mu(B_n)=2^{-n}$, $a_{n|\Gamma\times B_n}$ is free and $a_{n|\Gamma\times (X\setminus B_n)}$ is trivial for all $n\in \N$. Then, since
\[|\mu(A\cap C)-\mu((\gamma \cdot^{a_n} A)\cap C)|<2^{-n}\]
for all $\gamma\in \Gamma$ and all Borel $A,C\subseteq X$, we have $\undertilde{a_n}\rightarrow \undertilde{i_\Gamma}$  as $n\rightarrow \infty$ in $\undertilde{A}(\Gamma,X,\mu)$. However, since $\type(a_n)=2^{-n}\delta_{\set{e}}+(1-2^{-n})\delta_\Gamma$ for all $n\in \N$, it follows as in the  proof of \cref{contirs}, that $\type(\cind_\Gamma^\Delta(a_n))\nrightarrow \type(\cind_\Gamma^\Delta(i_\Gamma))$, when $n\rightarrow \infty$, in $\text{IRS}(\Delta)$. Thus $\undertilde{\cind}_\Gamma^\Delta$ cannot be continuous.
\end{proof}

By similar arguments as those above, we can also say something about the continuity of countable powers of an action. In general for an action $a\in A(\Gamma,X,\mu)$ and $n\in \N\cup\set{\N}$ the action $a^n\in A(\Gamma,X^n,\mu^n)$ is defined by
\[(\gamma \cdot^{a^n} f)(i)=\gamma\cdot^a f(i)\]
for all $f\in X^n$ and $i< n$. We then have the following result.
\begin{prop}\label{contcountprod} The map $\undertilde{a}\mapsto \undertilde{a}^{\mathbb{N}}$ from $\undertilde{A}(\Gamma,X,\mu)$ to $\undertilde{A}(\Gamma,X^\N,\mu^\N)$ is not continuous if $\Gamma\neq \set{e}$.
\end{prop}
\begin{proof}
Let $(a_n)_n\in A(\Gamma,X,\mu)$ and $(B_n)_n\subseteq X$ be as in the proof of \cref{contcind} and assume towards a contradiction that the map is continuous. Then we would have $\type(a_n^\N)\rightarrow \type(i_\Gamma^\N)$ as $n\rightarrow \infty$ in $\text{IRS}(\Gamma)$. But,
\[\type(a_n^\N)(\set{\Gamma})=\mu^\N (\prod_{m\in \N}(X\setminus B_n)) = 0\]
for all $n\in \N$, while $\type(i_\Gamma^\N)(\set{\Gamma})=1$.
\end{proof}
\noindent
\begin{remark}
In \cite[Theorem 1.2]{B18} it is shown that for a class of groups, containing the non-abelian free groups, the operation $\ul{a}\mapsto \ul{a}^2$ is not continuous, not even when restricted to the space of free weak equivalence classes.  As a corollary, for any group $\Gamma$ in this class, the map $$\ul{\cind}_\Gamma^{\Gamma\times (\Z/2\Z)}\colon \ul{A}(\Gamma,X,\mu)\rightarrow \ul{A}(\Gamma\times (\Z/2\Z),X^2,\mu^2)$$ is not continuous, again, not even when restricted to the space of free weak equivalence classes. So, while co-induction on weak equivalence classes is continuous in the finite index case, when the big group is amenable, this is not the case in general.
\end{remark}

\section{Properties of the co-induced invariant random subgroups}
We will study here different properties of the co-induced invariant random subgroups such as mixing properties and non-atomicity. We also obtain a characterization of when the co-induced action is free.

First note that it is clear that if $a\in \text{A}(\Gamma,X,\mu)$ is ergodic (resp., weakly mixing), then $\type(a)\in \irs(\Gamma)$ is ergodic (resp., weakly mixing). The converse does not hold in general. 
For example, $a$ can be a free non-ergodic action, but $\type(a)=\delta_{\set{e}}$ is ergodic.
However, for each $\theta\in \irs(\Gamma)$, if $\theta$ is ergodic (resp., weakly mixing), the action $a$ constructed in \cite[Proposition 13]{AGV14} which has $\type(a)=\theta$, will also be ergodic (resp., weakly mixing), see \cite[Theorem 5.11]{T-D15b}, provided $\theta$ concentrates on infinite index subgroups.

By use of the analogous result for actions, we obtain the following result. 

\begin{prop}\label{mixirs} Let $\Gamma\leq \Delta$ with $[\Delta : \Gamma]=\infty$. Then $\cind_\Gamma^\Delta(\theta)\in \irs(\Delta)$ is weakly mixing for any $\theta\in \irs(\Gamma)$.
\end{prop}

\begin{proof}
Let $\theta\in \irs(\Gamma)$ and let $a\in A(\Gamma,X,\mu)$ satisfy $\type(a)=\theta$. Then by \cref{mixact} we have $\cind_\Gamma^\Delta(a)$ is weakly mixing and hence, by \cref{coinducedIRS}, so is $\cind_\Gamma^\Delta(\theta)$.
\end{proof}

In general, weakly mixing is the strongest mixing property one can hope for a non-atomic IRS. Indeed, by the result of \cite{T-D15a}, if an IRS $\theta$ is totally ergodic (i.e., the restriction of the conjugacy action $\Gamma\ac(\sub(\Gamma),\theta)$ to any infinite subgroup of $\Gamma$ is ergodic), then there is a finite normal subgroup $\Lambda\unlhd\Gamma$ such that $\theta(\sub(\Lambda))=1$.  

Let $\Gamma$ be any group and $\theta\in \irs(\Gamma)$. We let 
$\ker(\theta)=\set{\gamma\in \Gamma\mid \theta(N_\gamma^\Gamma)=1}.$
Note that $\ker(\theta)$ is a subgroup of $\Gamma$. 

\begin{prop}\label{atomicform} Let $\Gamma\leq \Delta$ with $[\Delta :\Gamma]=\infty$ and $\theta\in \irs(\Gamma)$. If $\cind_\Gamma^\Delta(\theta)$ is atomic, then $\cind_\Gamma^\Delta(\theta)=\delta_{\core_\Delta(\ker(\theta))}$.
\end{prop}
\begin{proof}Assume $\Lambda\in \sub(\Delta)$ satisfies $\cind_\Gamma^\Delta(\set{\Lambda})>0$. Then the orbit of $\Lambda$ must be finite and $\cind_\Gamma^\Delta$ restricted to this orbit is a uniform measure. The diagonal of the orbit is then a fixed positive measured subset of $\sub(\Delta)^2$, which is invariant under the diagonal action of $\Delta$. Thus, as $\cind_\Gamma^\Delta(\theta)$ is weakly mixing, the orbit must be $\set{\Lambda}$ and $\cind_\Gamma^\Delta(\theta)(\set{\Lambda})=1$. So $\cind_\Gamma^\Delta(\theta)=\delta_{\Lambda}$ and since $\Lambda\subseteq \core_\Delta(\Gamma)$, we have for some transversal $T\subseteq \Delta$ of $\Delta/\Gamma$ that
\[\gamma\in \Lambda \iff \prod_{t\in T}\theta(N_{t^{-1}\gamma t}^\Gamma) = 1 \iff \gamma\in \core_\Delta(\ker(\theta)).\]
Thus $\cind_\Gamma^\Delta(\theta)=\delta_{\core_\Delta(\ker(\theta))}$, as wanted.
\end{proof}

From the proposition above, it follows that in order to obtain a non-atomic IRS via the co-induction operation, we just need to ensure that we do not obtain a Dirac measure. Thus we have the following criterion.
\begin{cor}\label{contmeas} Let $\Gamma\leq \Delta$ with $[\Delta : \Gamma]=\infty$ and $\theta\in \irs(\Gamma)$. Moreover, let $T\subseteq \Delta$ be a transversal for the left cosets in $\Delta/\Gamma$. Then $\cind_\Gamma^\Delta(\theta)$ is non-atomic if and only if there is $\gamma\in \core_\Delta(\Gamma)\setminus \core_\Delta(\ker(\theta))$ such that
\[\sum_{t\in T}\left(1-\theta(N_{t^{-1}\gamma t}^\Gamma)\right)<\infty\]
and $\theta(N_{t^{-1}\gamma t}^\Gamma)>0$ for all $t\in T$.
\end{cor}

\begin{proof} First note that $\cind_\Gamma^\Delta(\theta)$ is weakly mixing by \cref{mixirs}. Hence it follows by \cref{atomicform} that $\cind_\Gamma^\Delta(\theta)$ is non-atomic if and only if $\cind_\Gamma^\Delta(\theta)\neq \delta_{\core_\Delta(\ker(\theta))}$, and the latter is equivalent to the statement in the corollary. 

\end{proof}

Note that if $\theta=\type(a)$ for some $a\in A(\Gamma,X,\mu)$, then we have $\theta(N_{t^{-1}\gamma t}^\Gamma)=\mu((\fix_a(t^{-1}\gamma t))$. Thus 
\[\sum_{t\in T}\left(1-\theta(N_{t^{-1}\gamma t}^\Gamma)\right)=\sum_{t\in T}d_u((t^{-1}\gamma t)^a, 1),\]
where $(t^{-1}\gamma t)^a$ denotes the automorphism in $\Aut(X,\mu)$ induced by $t^{-1}\gamma t$.

It is clear that if $\Gamma\leq \Delta$ and  $\Lambda\trianglelefteq \Delta$ is normal with $\Lambda\subseteq \Gamma$, we have that $\delta_\Lambda\in \irs(\Gamma)$ satisfies $\cind_\Gamma^\Delta(\delta_\Lambda)=\delta_\Lambda$. Thus all possible Dirac measures in $\irs(\Delta)$ are contained in the image of the co-induction operation. In some cases these are the only ones.

\begin{prop} If $\Z\leq \Delta$ with $[\Delta : \Z]=\infty$, then any $\theta\in {\rm IRS}(\Z)$ satisfies $\cind_\Z^\Delta(\theta)=\delta_{\core_\Delta(\ker(\theta))}$.
\end{prop}
\begin{proof}
Let $n\in \N$ satisfy $n\Z=\core_\Delta(\Z)$. The for each $t\in T$ and $m\in n\Z$ we have $t^{-1}mt=\pm m$. Thus, since $\theta(N_m^\Z)=\theta(N_{-m}^\Z)$ we have
$\cind_\Z^\Delta(\theta)\left(N_m^\Delta\right)=1$ if $m\in \ker(\theta)$ and $\cind_\Z^\Delta(\theta)\left(N_{m}^\Delta\right)=0$, if $m\notin \ker(\theta)$, as wanted.
\end{proof}

In such cases, this means that the co-induced action has almost everywhere fixed stabilizers and the co-induced action of a faithful action is free. In general it is easily seen that if an action is free, then so are all its co-induced actions. By use of the description of the co-induced IRS in \cref{coinducedIRS}, we obtain the following complete characterization of when the co-induced action is free. 

\begin{prop} Let $\Gamma\leq \Delta$, $T\subseteq \Delta$ a transversal for the left cosets in $\Delta/\Gamma$ and $a\in \act(\Gamma,X,\mu)$. Then $\cind_\Gamma^\Delta(a)$ is not free if and only if for some $\gamma\in \core_\Delta(\Gamma)\setminus \set{e}$  we have
\[\sum_{t\in T}\left(1-\mu(\fix_a(t^{-1}\gamma t )) \right)<\infty\]
and $\mu\left(\fix_a(t^{-1}\gamma t)\right)>0$ for all $t\in T$.
\end{prop}
\begin{proof}
It follows directly from  \cref{coinducedIRS} that $\cind_\Gamma^\Delta(a)$ is not free if and only if for some $\gamma\in \core_\Delta(\Gamma)\setminus \set{e}$ we have
\[\prod_{t\in T}\type(a)(N_{t^{-1}\gamma t}^\Gamma) = \prod_{t\in T}\mu\left((\fix_a(t^{-1}\gamma t)\right)>0.\]
Since the latter is equivalent to the statement in the proposition, the conclusion follows.
\end{proof}

Note that for any $\gamma\in \Gamma$ we have $$1-\mu(\text{Fix}_a(\gamma))=d_u(\gamma^a, 1).$$
So for the co-induced action to be non-free, in the case $[\Delta :\Gamma]=\infty$, the conjugates of some $\gamma\in \core_\Delta(\Gamma)\setminus \set{e}$ under the transversal $T$ must uniformly converge very fast to the identity in $\text{Aut}(X,\mu)$. 
\begin{remark}
For any group $\Gamma$ there exists a group $\Delta$ such that $\Gamma\leq \Delta$ and $\core_\Delta(\Gamma)=\set{e}$. Thus for such groups any action $a\in \text{A}(\Gamma,X,\mu)$ will satisfy that $\cind_\Gamma^\Delta(a)$ is free.
\end{remark}

\section{New constructions of non-atomic, weakly mixing invariant random subgroups}

In this section we will apply the co-induction operation on invariant random subgroups to construct new examples of continuum size families consisting of non-atomic, weakly mixing invariant random subgroups on several classes of groups.

\subsection{A sufficient criterion}
We will provide in this subsection a sufficient criterion for an infinite index subgroup to generate continuum many non-atomic, weakly mixing co-induced invariant random subgroups on the bigger group. 

In the following, for a group $\Gamma$ and a subset $S\subseteq \Gamma$, we let $\langle S\rangle_\Gamma$  denote the subgroup generated by $S$ in $\Gamma
$ and  $\langle\langle S\rangle\rangle_\Gamma$ denote the normal subgroup generated by $S$ in $\Gamma$.
\begin{prop}\label{impli} Let $\Gamma\leq \Delta$ with $[\Delta :\Gamma]=\infty$. Consider the statements:
\begin{itemize}
\item[(1)] There exists a transversal $T=\set{t_i\mid i \in \N}$ for the left cosets in $\Delta/\Gamma$ and $\gamma_0\in \core_\Delta(\Gamma)$ such that the chain of normal subgroups $(\ol{\Gamma}_{k,T,\gamma_0})_{k\in \N}$, given by
\[\ol{\Gamma}_{k,T,\gamma_0}=\langle\langle t_i^{-1}\gamma_0t_i \mid i\geq k\rangle\rangle_\Gamma,\]
is not constant.
\item[(2)] There exists a continuum size family $(\theta_i)_{i\in I}\in \irs(\Gamma)$ such that we have $\left(\cind_\Gamma^\Delta(\theta_i)\right)_{i\in I}\in \irs(\Delta)$  are all  non-atomic, weakly mixing and satisfy $\cind_\Gamma^\Delta(\theta_i)\neq \cind_\Gamma^\Delta(\theta_j)$  for all $i,j\in I$ with $i\neq j$.
\item[(3)] There exists $\theta\in \irs(\Gamma)$ such that $\cind_\Gamma^\Delta(\theta)\in \irs(\Delta)$ is non-atomic.
\item[(4)] There exists $\theta\in \irs(\Gamma)$ such that $\cind_\Gamma^\Delta(\theta)\in \irs(\Delta)$ is not a Dirac measure.
\item[(5)]
For any transversal $T=\set{t_i\mid i \in \N}$ for the left cosets in $\Delta/\Gamma$ there is $\gamma_0\in \core_\Delta(\Gamma)$ such that the chain of subgroups $(\Gamma_{k,T,\gamma_0})_{k\in \N}$, given by
\[\Gamma_{k,T,\gamma_0}=\langle t_i^{-1}\gamma_0t_i \mid i\geq k\rangle_\Gamma,\]
is not constant.

\end{itemize}
Then
$(1)\implies (2)\implies (3)\implies (4)\implies (5).$
\end{prop}

\begin{proof}
It is clear that $(2)\implies (3)\implies (4)$. Thus it suffices to prove $(1)\implies (2)$ and $(4)\implies (5)$.

For the implication $(1)\implies (2)$, assume $(1)$ holds for $T$ and $\gamma_0$.
We will first construct one such invariant random subgroup. Afterwards we will argue how to obtain uncountably many.

Let $\theta=\sum_{k\in \N} 2^{-k-1}\delta_{\ol{\Gamma}_{k,T,\gamma_0}}$. Then the non-constant assumption on the sequence $(\ol{\Gamma}_{k,T,\gamma_0})_{k\in \N}$ ensures that $\gamma_0\notin \core_{\Delta}(\ker (\theta))$, as for some $j,k\in \N$ with $j\leq k$ we have $t_j^{-1}\gamma_0t_j\notin \ol{\Gamma}_{k,T,\gamma_0}$. Moreover, it follows directly by the construction of $\theta$ that $\theta(N_{t_i^{-1}\gamma_0t_i}^\Gamma)>0$ for all $i\in \N$ and that $$\sum_{i\in \N}\left(1-\theta(N_{t_i^{-1}\gamma_0t_i}^\Gamma)\right)<\infty.$$ So the assumptions of \cref{contmeas} are satisfied and thus the co-induced measure must be non-atomic. Since $[\Delta \colon \Gamma]=\infty$, it will also be weakly mixing.

Now to construct uncountably many of these, let $N\in \N$ be least such that $\ol{\Gamma}_{N+1,T,\gamma_0}\subsetneq \ol{\Gamma}_{N,T,\gamma_0}$ and let $\lambda=\sum_{k\leq N+1}2^{-k-1}$. Next, fix $S\subseteq \set{0,\ldots, N}$ such that for $k\in \N$ we have $t_k^{-1}\gamma_0t_k\notin \ol{\Gamma}_{N+1,T,\gamma_0}$ if and only if $k\in S$.
For each $a\in(0,\lambda)$ put $$\theta_{a}= a\delta_{\ol{\Gamma}_{0,T,\gamma_0}}+(\lambda-a)\delta_{\ol{\Gamma}_{N+1,T,\gamma_0}}+\sum_{N+1<k}2^{-k-1}\delta_{\ol{\Gamma}_{k,T,\gamma_0}}.$$
 Then we have $$\theta_a(N_{t_s^{-1}\gamma_0t_s}^\Gamma)=a$$ for all $s\in S$, while
$$\theta_a(N_{t_k^{-1}\gamma_0t_k}^\Gamma)=\theta(N_{t_k^{-1}\gamma_0t_k}^\Gamma)$$
for all $k\notin S$. Thus, by the description of  the co-induction operation given in  \cref{coinducedIRS}, we obtain that
\[\cind_\Gamma^\Delta(\theta_a)\left((N_{\gamma_0}^{\Delta}\right)=\prod_{k\in \N}\theta_a(N_{t_k^{-1}\gamma_0t_k}^\Gamma) = a^{|S|} \prod_{k\in \N\setminus S}\theta(N_{t_k^{-1}\gamma_0t_k}^\Gamma).\]
So $\left(\cind_\Gamma^\Delta(\theta_a)\right)_{a\in (0,\lambda)}$ is a continuum size family of non-atomic, weakly mixing invariant random subgroups of $\Delta$, as wanted.

For the implication $(4)\implies (5)$, assume  that $(4)$ holds.  Let $T=\set{t_i\mid i\in \N}$ be a transversal for $\Delta/\Gamma$. We have that
\[\cind_\Gamma^\Delta(\theta)\left(N_F^\Delta\right)=\prod_{t\in T}\theta(N_{t^{-1}Ft}^\Gamma)\]
for all $F\subseteq \core_\Delta(\Gamma)$ finite.
Since $\cind_\Gamma^\Delta(\theta)$ is not a Dirac measure, there exists $\gamma_0\in \core_\Delta(\Gamma)$ such that $\cind_\Gamma^\Delta(\theta)\left(N_{\gamma_0}^\Delta\right)\in (0,1)$. Thus, if we let $b\in \text{A}(\Gamma,X,\mu)$ satisfy that $\type(b)=\theta$, we have $\mu(\fix_b(t_{m}^{-1}\gamma_0t_m))=\lambda_m<1$ for some $m\in \N$  and
\[\sum_{i\in \N}\left(1-\mu\left(\fix_b(t_i^{-1}\gamma_0t_i)\right)\right) <\infty.\]
By convergence of the series, it follows that for some $N\in \N$ we have
\[\mu\left(\set{ x\in X\mid (\forall \gamma \in \Gamma_{N,T,\gamma_0}) \gamma \cdot^b x= x}\right)>\lambda_m.\]
Thus we must have $t_m^{-1}\gamma_0t_m \notin \Gamma_{N,T,\gamma_0}$, as wanted.
\end{proof}
\noindent

\begin{remark}\label{inverstrans} In general, if $\Gamma\leq \Delta$ is a normal subgroup, then $T$ is a transversal for the left cosets in $\Delta/\Gamma$ if and only if $T^{-1}=\set{t^{-1}\mid t\in T}$ is a transversal, as well. Thus in this case, the statement
\begin{itemize}
\item[\textit{(1')}] \textit{There exists a transversal $R$ for the left cosets in $\Delta/\Gamma$ and $\gamma_0\in \Gamma$ such that the chain of normal subgroups $(\ol{\Lambda}_{k,R,\gamma_0})_{k\in \N}$ given by 
\[\ol{\Lambda}_{k,R,\gamma_0}=\langle\langle r_i\gamma_0r_i^{-1}\mid i\geq k\rangle\rangle_\Gamma\]
is not constant.}
\end{itemize}
is equivalent to condition $(1)$ in \cref{impli}. 
\end{remark}
\begin{remark} If $\Gamma$ in \cref{impli} is abelian, then all the statements are equivalent. 
We also point out that the invariant random subgroups constructed in the proof of $(1)\implies (2)$ above, are not weakly mixing when restricted to $\Gamma$.
\end{remark}
\subsection{Wreath products and HNN-extensions}

We will here apply the criterion in \cref{impli} to wreath products and HNN-extensions. 

\medskip
(1) Let $G,H$ be countable groups and consider the  action $G\ac^{\alpha}\oplus_{G} H$ given by $g\cdot^{\alpha} f(g_0)=f(g^{-1}g_0)$. The wreath product of $G$ by $H$ is then the semidirect product $(\oplus_G H)\rtimes_{\alpha} G$ and is denoted by $H\wr G$.
\medskip

{\it Construction of continuum many non-atomic, weakly mixing invariant random subgroups on $H\wr G$, for $G,H$ countable groups such that $G$ is infinite and $H$ is not trivial.}


\medskip
Let $\Gamma=\oplus_GH$ and $\Delta=H\wr G$. Then $\Gamma\unlhd \Delta$ and $G\subseteq \Delta$ is a transversal for the left cosets $\Delta/\Gamma$. Fix an enumeration $G=\set{g_i\mid i\in \N}$ and let $h_0\in H\setminus \set{e_H}$. Define $\gamma_0\in \Gamma$ by
\begin{equation*}
\gamma_0(i)=\begin{cases}
e_H \qquad &\text{if} \quad i\neq 0\\
h_0 \qquad &\text{if} \quad i=0
\end{cases}.
\end{equation*}
Then, as $\left(\ol{\Gamma}_{k,G,\gamma_0}\right)_{k\in \N}$ is not constant, following \cref{impli} we construct continuum many non-atomic, weakly mixing invariant random subgroups on $H\wr G$.

If $\Omega$ is a countable set and we have an action $G\ac^\alpha \Omega$, we may form a wreath product by letting $G\ac^{\ol{\alpha}}\oplus_{\Omega} H$ be given by $(g\cdot^{\ol{\alpha}} f)(w)=f(g^{-1}\cdot^\alpha w)$ and then consider the semidirect product $(\oplus_{\Omega}H)\rtimes_{\ol{\alpha}}G$. We denote such a wreath product by $H\wr_\Omega G$. Arguments similar to those in the preceding paragraph work as well for  $H\wr_\Omega G$, if the action $G\ac^\alpha \Omega$ has an infinite orbit with finite stabilizers.

\medskip
(2) Next we will consider HNN-extensions over ``small'' subgroups.

\medskip
{\it Construction of continuum many non-atomic, weakly mixing invariant normal subgroups for the HNN extension $G=\langle H,t \mid t^{-1}at=\phi(a), a\in A\rangle$, where $H$ is a countable group, $A\leq H$ and $\phi\colon A\rightarrow H$ is an embedding with
$\langle\langle A\cup \phi(A)\rangle\rangle \neq H.$}

\medskip
Let $H_n=\set{h_n\mid h\in H}$ for each $n\in \Z$ be a copy of $H$ and put
\[F=\langle \ast_{n\in \Z}H_n \mid (\forall j\in \Z)(\forall a\in A) a_{j+1}=\phi(a)_j\rangle. \]
Then $G\cong F\rtimes_\psi \Z$, where $$\psi(h^1_{i_1}h^2_{i_2}\cdots h^k_{i_k})=h^1_{i_1+1}h^2_{i_2+1}\cdots h^k_{i_k+1}$$ for all $h^1,\ldots, h^k\in H$ and $i_1,\ldots,i_k\in \Z$ (see \cite[Theorem 17.1]{B08}). Now let $\Lambda=\langle\langle A\cup \phi(A)\rangle\rangle$ and consider the homomorphism $f\colon F\rightarrow H/\Lambda$ induced by the homomorphisms $f_i\colon H_i\rightarrow H/\Lambda$ for $i\in \Z$ given by $f_i(h_i)=e\Lambda$ if $i\neq 0$ and $f_0(h_0)= h\Lambda$ for all $h\in H$.
For a fixed $x\in H\setminus \Lambda$ we then have  $f(x_0)\neq e$, while $f(x_i)=e$ for all $i\neq 0$. Thus $$x_0\notin \langle\langle \psi^i(x_0)\mid i\in \Z\setminus \set{0}\rangle\rangle =\langle\langle x_i \mid i\in \Z\setminus \set{0}\rangle\rangle$$ and hence using \cref{impli} we construct continuum many non-atomic, weakly mixing invariant random subgroups in $G$.

Note that this covers the case where $\phi$ is an automorphism of a non-trivial normal subgroup of $H$. We also have the following application.

\begin{cor} If $n,m\in \Z\setminus\{0\}$ are not relatively prime, then there are continuum many non-atomic, weakly mixing invariant random subgroups on $BS(n,m)=\langle x,t\mid tx^nt^{-1}=x^m\rangle$.
\end{cor}

\begin{proof}
We have that $BS(n,m)$ is the HNN-extension of $\Z$ with respect to the isomorphism $\phi\colon n\Z\rightarrow m\Z$ and $\langle\langle n\Z\cup m\Z\rangle\rangle = gcd(n,m)\Z$.
\end{proof}

\subsection{Non-abelian free groups}
We will now turn our attention towards the non-abelian free groups. It follows already from the results in \cite{BGK17} that these groups admit continuum many non-atomic, weakly mixing invariant random subgroups. In this subsection we show how the co-induction can be used to give alternative constructions of invariant random subgroups with these properties.

\medskip
(1) First we will use the co-induction operation from $\F_\infty$ to various semi-direct products of the form $\F_\infty \rtimes_\phi \Z$, where $\phi$ is induced by a permutation of the generators of $\F_\infty$, to construct new invariant random subgroups on $\F_\infty$.

\medskip
{\it Construction of continuum many non-atomic, weakly mixing invariant random subgroups on $\F_\infty$.}

\medskip
Fix $\F_\infty=\langle b, a_k\mid k\in \N\rangle$ and let $(B_k)_k\subseteq X$ be a Borel partition with $\mu(B_k)=2^{-k-1}$. For each $k\in \N$, put $A_k=B_0\cup B_1\cup\cdots \cup B_k$ and fix a free action $\F_\infty\ac^{\alpha_k}B_k$. Then define $\alpha\in A(\F_\infty,X,\mu)$ by $a_k\cdot^\alpha x=x$ if $x\in A_k$ and $a_k\cdot^{\alpha} x =a_k\cdot^{\alpha_j}x$ if $x\in B_j$ for some $j>k$. Finally, let $b$ act as a weakly mixing transformation on $X$ to ensure that $\alpha$ is weakly mixing.

Next, let $S\subseteq \N$ be infinite and let $\pi_S\colon \N\rightarrow \N$ be a permutation which is transitive on $S$ and fixes every element of $\N\setminus S$.  We then define $\phi_S \colon \Z\rightarrow \text{Aut}(\F_\infty)$ by  $\left(\phi_S(z)\right)(b)=b$ and $\left(\phi_S(z)\right)(a_k)=a_{\pi_S^{z}(k)}$  for all $z\in \Z$ and $k\in \N$. Consider $\Delta_S=\F_\infty\rtimes_{\phi_S}\Z$ and let 
\[\theta_S=\type\left(\cind_{\F_\infty}^{\Delta_S}(\alpha)_{|\F_\infty}\right)=\left(\cind_{\F_\infty}^{\Delta_S}(\type(\alpha))\right)_{|\sub(\F_\infty)}\in \irs(\F_\infty).\]
Note that by \cref{mixact} we have $\theta_S$ is weakly mixing and for $k\in \N$ we have
\begin{equation*}
\theta_S(N_{a_k}^{\F_\infty})=\begin{cases}
\prod_{z\in \mathbb{Z}} \mu\left(\text{Fix}_\alpha(a_{\pi_S^z(k)})\right) \qquad &\text{if} \quad k\in S\\
\prod_{z\in \mathbb{Z}} \mu\left(\text{Fix}_\alpha(a_{k})\right) \qquad &\text{if} \quad k\notin S\\
\end{cases}.
\end{equation*}
So for $k\in \N$ it holds that \[\theta_S(N_{a_k}^{\F_\infty})\in (0,1) \iff k\in S\] and hence 
$\theta_S$ is non-atomic. Moreover, this implies that whenever $S,T\subseteq \N$ are infinite with $S\neq T$ we have $\theta_S\neq \theta_T$.

\medskip
(2) Next consider another construction using co-induction, which allows us to construct continuum many  non-atomic, weakly mixing invariant random subgroups on every non-abelian free group. 

\medskip
{\it Construction of continuum many non-atomic, weakly mixing invariant random subgroups on $\F_n$ for $n\in \N\cup\{\infty\}$ with $n\geq 2$.}

\medskip
Fix $n\in \N\cup\set{\infty}$ with $n\geq 2$ and some free generators $\F_n=\langle a_i\mid i< n\rangle$. Consider the surjective group homomorphism $\phi\colon \F_n\rightarrow \Z$ given by $\phi(a_i)=0$ for $0<i<n$ and $\phi(a_0)=1$. Then let $\Gamma=\ker(\phi)$ and note that
\[\Gamma=\langle a_0^{-k}a_ia_0^k\mid k\in \Z, 0<i<n\rangle.\]
This set freely generates $\Gamma$ as a copy of $\F_\infty$ inside $\F_n$.
Moreover, the set $T=\set{a_0^k\mid k\in \Z}$ constitutes a transversal for $\F_n/\Gamma$.

Now for each $\lambda\in (0,1)$ let $(T_{k}^\lambda)_{k\in \Z}\in \Aut(X,\mu)$ satisfy that the action induced by $\langle T_{-1}^\lambda,T_0^\lambda,T_1^\lambda \rangle $ is weakly mixing and for each $k\in \Z$ we have 
\begin{equation*}
\mu(\text{Fix}(T_k^\lambda))=\begin{cases}
3^{-1} \qquad &\text{if} \quad |k|< 2 \\
\lambda \qquad &\text{if} \quad |k|=2\\
1-2^{-k-1} \qquad &\text{if} \quad |k|>2\\
\end{cases}.
\end{equation*}
One way to choose $T_{-1}^\lambda ,T_0^\lambda, T_1^\lambda$ is to decompose $X=X_{-1}\sqcup X_0\sqcup X_1$ such that $\mu(X_{-1})=\mu(X_0)=\mu(X_1)=3^{-1}$ and then let $T_j$ be weakly mixing when restricted to $X\setminus X_j$ and trivial on $X_j$ for $j\in \set{-1,0,1}$.
Next, define an action $\alpha_\lambda\in A(\Gamma ,X,\mu)$ by letting   $a_0^{-k}a_1a_0^k \cdot^{\alpha_\lambda} x = T_k^{\lambda}(x)$ and $a_0^{-k}a_ia_0^k \cdot^{\alpha_\lambda} x = x$ for all $1<i<n$ and $k\in \Z$. Then put $\theta_\lambda =\type(\alpha_\lambda)$. Note that all conditions of \cref{contmeas} are satisfied with respect to $a_1\in \Gamma$ and we have that
\[\cind_\Gamma^{\F_n}(\theta_\lambda)(N_{a_1}^{\F_n})=\prod_{k\in \Z}\theta_\lambda(N_{a_0^{-k}a_1a_0^k}^\Gamma)=\lambda^23^{-3}\prod_{k\in \Z\setminus \set{-2,\ldots,2}}(1-2^{-k-1}).\]
Thus $\cind_\Gamma^{\F_n}(\theta_{\lambda_0})\neq \cind_\Gamma^{\F_n}(\theta_{\lambda_1})$ for all $\lambda_0,\lambda_1\in (0,1)$ with $\lambda_0\neq \lambda_1$.

\noindent
\begin{remark} Using an action similar to the one in the construction above one can give a proof of the following algebraic fact: Let $s,w_1,w_2, \ldots \in \F_\infty$ satisfy $w_i^{-1}sw_i \neq w_j^{-1}sw_j$ if $i\neq j$. Then the sequence
\[ \set{ w_n^{-1}sw_n \mid n\in \N }\]
does not extend to a basis of $\F_\infty$. Indeed, assume towards a contradiction, that we may extend the sequence to a basis of $\F_\infty$. Then we would have that $\set{ w_n^{-1}sw_n \mid n\in \N }$ generates a copy, $F_\infty$, of $\F_\infty$ as a subgroup of $\F_\infty$. So let $a\in \act(F_\infty,X,\mu)$ be an action such that $d_u((w_n^{-1}sw_n)^a,1)\neq 0$ for all $n\in \N$ and
\[d_u((w_n^{-1}sw_n)^a, 1)\rightarrow 0\]
as $n\rightarrow \infty$ in $\text{Aut}(X,\mu)$.
Now, since the sequence extends to a basis, we may extend this action to an action $b\in A(\F_\infty,X,\mu)$. Therefore we would have
\[d_u((w_n^{-1}sw_n)^a,1) =d_u((w_n^{-1}sw_n)^b,1)= d_u((w_n^{-1})^bs^bw_n^b,1)=d_u(s^b,1),\]
which contradicts the convergence above.
\end{remark}

\subsection{Free products with normal amalgamation}
Here we will use co-induction to give constructions of non-atomic, weakly mixing invariant random subgroups on certain free products of groups with normal amalgamation. Other constructions can be found in \cite{BGK17} but our proofs use completely different and elementary methods. First we will consider free products without amalgamation.

\medskip
{\it Construction of continuum many non-atomic, weakly mixing invariant random subgroups on $G\ast H$, where $G$, $H$ are non-trivial countable groups with $H$ infinite, with support in
\[\Gamma=\langle [g,h]\mid g\in G, h\in H\rangle.\] Moreover, we can ensure that these invariant random subgroups are weakly mixing, when restricted to $\Gamma$.}

\medskip
Let $\Delta= G\ast H$, consider the homomorphism $\phi\colon \Delta\rightarrow G\times H$ induced by the homomorphisms $g\mapsto (g,e_H), h\mapsto (e_G, h)$, and put $\Gamma=\ker \phi$. Then $\Gamma$ is freely generated by the commutators:
\[\Gamma=\langle [g,h]\mid g\in G\setminus \set{e_G}, h\in H\setminus\set{e_H}\rangle,\]
where $[g,h]=ghg^{-1}h^{-1}$ and $T=\set{ gh \mid g\in G, h\in H}$ is a transversal for the left cosets $\Delta/\Gamma$. Now fix $g_0\in G\setminus \set{e_G}$ and $h_0\in H\setminus \set{e_H}$. For each $\lambda\in (0,1)$, let $a_\lambda\in A(\Gamma,X,\mu)$ be an action satisfying $\mu(\fix_{a_\lambda}([g_0,h_0]))=\lambda$ and $\mu(\fix_{a_\lambda}([g,h]))=1$ for all $g\in G\setminus \set{g_0}$ and $h\in H\setminus \set{h_0}$. Let $\theta_\lambda=\type(a_\lambda)$ and note that
\[\cind_\Gamma^\Delta(\theta_\lambda)\left(N_{[g_0,h_0]}^\Delta\right)=\prod_{gh\in T}\theta_\lambda\left(N_{h^{-1}g^{-1}[g_0,h_0]gh}^\Gamma\right).\]
We have
\[h^{-1}g^{-1}[g_0,h_0]gh=[h^{-1},g^{-1}g_0][g^{-1}g_0,h^{-1}h_0][h^{-1}h_0,g^{-1}][g^{-1},h^{-1}]\]for all $g\in G$, $h\in H$. Thus $[g_0,h_0]$ or its inverse is in the reduced word over the alphabet of non-trivial commutators of
$$h^{-1}g^{-1}[g_0,h_0]gh$$ if and only if
\[(g,h)\in \set{(e_G,h_0^{-1}),(e_G,e_H), (g_0^{-1},e_H),(g_0^{-1},h_0^{-1})}.\]
Therefore $\cind_\Gamma^\Delta(\theta_\lambda)\left(N_{[g_0,h_0]}^\Delta\right)=\lambda^4$ and so $\left(\cind_\Gamma^\Delta(\theta_\lambda)\right)_{\lambda\in (0,1)}$ constitute a continuum size family of non-atomic, weakly mixing invariant random subgroups of $\Delta$.

To ensure weakly mixing when restricted to $\Gamma$, let $h_1,h_2,h_3\in H$ satisfy that  $h_0,h_1,h_2,h_3$ are distinct and that $h_jh_0\neq h_i$, $h_jh_0^{-1}\neq h_i$ for all $i,j\in \set{0,1,2,3}$. Then modify $a_\lambda$ such that the action of $[g_0,h_1], [g_0,h_2], [g_0,h_3]$ is weakly mixing and each satisfies $\mu(\fix_{a_\lambda}([g_0,h_{i}]))=1/3$ for $i\leq 3$. Note that the relation constrains on $h_0,h_1,h_2,h_3$ ensure that at most one of $[g_0,h_{i}]$ satisfies that it or its inverse is in the word
\[[h^{-1},g^{-1}g_0][g^{-1}g_0,h^{-1}h_0][h^{-1}h_0,g^{-1}][g^{-1},h^{-1}],\]
when $g\in G$ and $h\in H$. Moreover, as with $[g_0,h_0]$, each will appear exactly four times. So we then have
\[\cind_\Gamma^\Delta(\theta_\lambda)\left(N_{[g_0,h_0]}^\Delta\right)=\lambda^43^{-12}.\]
Again, $\left(\cind_\Gamma^\Delta(\theta_\lambda)\right)_{\lambda\in (0,1)}$ constitute a continuum size family of non-atomic, weakly mixing invariant random subgroups of $\Delta$. These will now also be weakly mixing, when restricted to $\Gamma$ by \cref{mixact}.

\medskip

\begin{remark}
Gaboriau pointed out that in the paper D. Gaboriau and N. Bergeron,  Asymptotique des nombres de Betti, invariants $\ell^2$ et
laminations, {\it Comment. Math. Helv.}, {\bf 79(2)} (2004), 362--395, 2004, the following result is proved:
Let $G$ and $H$ be residually finite, infinite groups such that either
$b_1(G) - b_1^{(2)}(G)\not=1$ or 
$b_n(G) - b_n^{(2)}(G)\not=0$ for $n\geq 2$.
Then the free product $G\ast H$ admits continuum many IRS. These IRS are distinguished by their $\ell^2$-Betti numbers. 
(Here, $b_n(G)$ is the $n$-th Betti number of $G$ while  $b_n^{(2)}(G)$ is its $n$-th $\ell^2$-Betti number.)

\end{remark}

Next, note the following well-known simple fact.

\begin{prop}\label{epi} 
Let $\Delta,\Gamma$ be countable groups and $\phi\colon \Delta\rightarrow \Gamma$ a surjective group homomorphism. Then there is an embedding $\Psi\colon \irs(\Gamma)\rightarrow \irs(\Delta)$ such that if $\theta\in \irs(\Gamma)$ is ergodic, weakly mixing or non-atomic, so is $\Psi(\theta)$.
\end{prop}
\begin{proof} Note that the map
$\Phi\colon\sub(\Gamma)\rightarrow \sub(\Delta)$ given by $\Phi(\Lambda)=\phi^{-1}(\Lambda)$ is a homeomorphism with image
\[\Phi(\sub(\Gamma))=\set{\Lambda\in \sub(\Delta)\mid \ker(\phi)\subseteq \Lambda}.\]
Moreover, we have 
\[\Phi\left(\phi(\delta)\Lambda\phi(\delta)^{-1}\right) = \delta\Phi(\Lambda)\delta^{-1}\]
for all $\Lambda\in \sub(\Gamma)$ and $\delta\in \Delta$. So let $\Psi\colon \irs(\Gamma)\rightarrow\irs(\Delta)$ be given by $\Psi(\theta)=\Phi_*\theta$. It is then clear that $\Psi(\theta)$ is ergodic, weakly mixing or non-atomic if $\theta$ is. Since 
\[\Psi(\theta)(N_F^\Delta)=\theta(N_{\phi(F)}^\Gamma),\]
it follows by \cref{convbasic} that $\Psi$ is continuous.
\end{proof}

Now by use of the previous construction for free products and \cref{epi}, we can construct continuum many non-atomic, weakly mixing invariant random subgroups for the groups $G\ast_AH$, where $G, H$ and $A$ are countable groups satisfying that $A\unlhd G,H$ with $G/A$ non-trivial and $H/A$ infinite. 
This follows directly from \cref{epi} applied to the surjective group homomorphism $\phi\colon G\ast_AH\rightarrow G/A\ast H/A$.

The same applies to all the groups  $\ast_{i\in \N}H_i$, where $(H_i)_{i\in \N}$ is a countable family of countable groups with $H_0$ infinite and $H_1$ non-trivial, by looking at the natural surjective group homomorphism $\phi\colon \ast_{i\in \N}H_i \rightarrow H_0\ast H_1$.

\subsection{Automorphism invariant random subgroups of the free group of rank two}

In this part we will use the co-induction operation to construct non-atomic invariant random subgroups on $\F_2$ which are invariant under the action of the full automorphism group, as well.  Moreover, these invariant random subgroups will be weakly mixing with respect to the action of the automorphism group. 

Fix a basis $\F_2=\langle a,b\rangle$. We think of an element of $\F_2$ as represented by the induced reduced word in the letters $\set{a,b,a^{-1},b^{-1}}$. Consider the automorphisms $\chi,\xi,\phi,\psi,\tau \in \text{Aut}(\F_2)$ given by 
\[\chi(a)=a, \qquad \chi(b)=b^{-1}, \qquad \xi(a)=a^{-1}, \qquad \xi(b)=b, \qquad \tau(a)=b, \]
\[\tau(b)=a, \qquad \phi(a)=ab, \qquad \phi(b)=b, \qquad \psi(a)=a \quad \text{and} \quad \psi(b)=ba.\]
Let $\text{Fr}_+(\phi,\psi)$ denote the set of automorphisms generated by using only $\phi$ and $\psi$ (and not $\phi^{-1},\psi^{-1}$). Then
\begin{align*}
R= \lbrace\rho, \ \rho\tau, \ \xi\sigma, \ &\xi\sigma\tau, \ \rho\xi, \ \rho\tau\xi, \  \xi\sigma\xi, \ \xi\sigma\tau \xi, \ \rho\chi, \ \rho\tau\chi, \ \xi\sigma\chi,\\
&\xi\sigma\tau\chi, \ \rho\xi\chi, \ \rho\tau\xi\chi, \ \xi\sigma\xi\chi, \ \xi\sigma\tau\xi\chi \mid \sigma,\rho\in \FR, \sigma\neq 1\rbrace
\end{align*}
is a set of representatives for the left cosets in $\Aut(\F_2)/\F_2$, where $\F_2$ is identified with the subgroup of inner automorphisms (see \cite[Section 3]{CMZ81}). Note that $1\in \FR$ denotes the identity map.

Consider the word $w=aba^2b^2a^3b^3\cdots a^nb^n$ for some  $n>101$. The first goal is to prove that the family
\[\set{\eta(w)\mid \eta\in R}\]
satisfies the $C'(1/6)$ cancellation property. Recall that a subset of words $S\subseteq \F_2$ has the \textbf{$C'(1/6)$ cancellation property} if the set $\tilde{S}$ of all cyclically reduced cyclic conjugates of the words in $S$ and their inverses satisfies that if $u\in \F_2$ is an initial segment of $x,y\in \tilde{S}$ with $x\neq y$, then 
$$|u|<\frac{1}{6}\min \set{|x|,|y|}.$$
Here $|\cdot |$ denotes the length of a word in $\F_2$.  In case $S$ satisfies the $C'(1/6)$ cancellation property and $z\in \langle\langle S\rangle\rangle\setminus \tilde{S}$ is a cyclically reduced word, then there is $x\in \tilde{S}$ such that $|x|< |z|$.
For a proof of this see for example \cite[Theorem 4.5 in Chapter V]{LS77}.
We will use this fact to ensure that condition $(1)$ in \cref{impli} is satisfied for $\F_2\leq \Aut(\F_2)$.

Put
\begin{align*}
w_0&=w=aba^2b^2\cdots a^nb^n\\
w_1&=\xi\chi(w)=a^{-1}b^{-1}a^{-2}b^{-2}\cdots a^{-n}b^{-n}\\
w_2&=\xi(w)=a^{-1}ba^{-2}b^2\cdots a^{-n}b^{n}\\
w_3&=\chi(w)=ab^{-1}a^2b^{-2}\cdots a^nb^{-n}\\
w_4&=\tau w=bab^2a^2\cdots b^na^n\\
w_5&=\tau \xi\chi(w)=b^{-1}a^{-1}b^{-2}a^{-2}\cdots b^{-n}a^{-n}\\
w_6&=\tau \xi(w)=b^{-1}ab^{-2}a^2\cdots b^{-n}a^{n}\\
w_7&=\tau \chi(w)=ba^{-1}b^2a^{-2}\cdots b^na^{-n}.
\end{align*}
and let $v_i=w_i^{-1}$ for all $0\leq i\leq 7$. 
Below we will use the following terminology. For two words $x,y\in \F_2$ a \textbf{cancellation} of $x$ and $y$ is a string $u\in \F_2$ which appears in the reduced cycles of both $x,y$. 
We say that $u$ is a \textbf{bad cancellation} of $x$ and $y$ if $$|u|\geq 1/6\min \set{\|x\|,\|y\|}.$$ Here $\|\cdot \|$ denote the length of the induced cyclically reduced word.
We call a cancellation for \textbf{maximal} if it cannot be extended.
The goal is then to prove that there is no bad cancellation between any pair of words in the set
\[B=\set{ \rho(w_i),\xi\sigma(w_i),\rho(v_i),\xi\sigma(v_i)\mid \sigma,\rho\in \FR, \ \sigma\neq 1, \ 0\leq i\leq 7}.\]
Let 
\begin{align*}
B_0&=\set{\rho(w_i),\rho(v_i)\mid \rho\in\FR, \ 0\leq i\leq 7}\\
B_1&=\set{\xi\sigma(w_i),\xi\sigma(v_i)\mid  \sigma\in\FR \setminus\set{1},\  0\leq i\leq 7}.
\end{align*}
 Then it suffices to prove that there is no bad cancellation among the words in $B_0$ and then prove that there cannot be any bad cancellation between a word from $B_0$ and a word from $B_1$. We will begin with the former. Most of our arguments for this are based on the following two lemmas.
 
For a word $x\in \F_2$, we let $\ol{x}\in \F_2$ denote the word obtained from $x$ by switching every negative power of $a$ and $b$ to be positive.
 
\begin{lem}\label{estimatecancellation}
Let $x,y\in\F_2$, $\rho\in \FR$ and let $q$ be a cancellation of $\rho(x),\rho(y)$. Assume $N\in \N$ satisfies that for any cancellation $c$ of $x,y$ the total number of $a$'s and the total number of $b$'s in $\ol{c}$ are both less than $N$. Then
\[ |q|\leq (N+2)\left(|\rho(a)| +|\rho(b)|\right).\]
\end{lem}
\begin{proof} First let $S\in \set{\phi,\psi}$ and $u\in \set{a,b}$ be such that $u=a\iff S=\psi$. Assume $q$ is a maximal cancellation of $S(x),S(y)$.
  Then there is a maximal cancellation $c$ of $x,y$ such that $q$ is equal to one of the following strings:
\[S(c), \quad uS(c), \quad S(c)u^{-1} \quad \text{or}\quad uS(c)u^{-1}.\]
This can be seen by considering the pre-images through $S$ of all possible strings of the form 
\[t_0t_1t_2q_0t_3t_4t_5,\]
where $t_0,\ldots, t_5\in \set{a,a^{-1},b,b^{-1}}$ and $q_0\in \F_2$ are such that $t_2q_0t_3=q$ and $t_0t_1t_2q_0t_3t_4t_5$ is a legal string in the reduced cycles induced by $S(x)$ and $S(y)$, together with the cases where $|q|\in \set{0,1}$, $\|x\|\leq |q|+3$ and $\|y\|\leq |q|+3$. The latter cases are handled by considering pre-images of strings as above where the relevant $t_i$'s are omitted.

Next, for $\rho\in \FR$, we let $S_0,\ldots ,S_N\in \set{\phi,\psi}$ and $u_0,\ldots, u_N\in \set{a,b}$ be such that $\rho=S_N\cdots S_0$ and
$u_i=a \iff S_i=\psi$, for $0\leq i\leq N$. Now  let $q$ be a maximal cancellation of $\rho(x),\rho(y)$. Then, by repeating the argument above, there is a maximal cancellation $c$ of $x,y$ such that
\begin{align*}
|q|&\leq |u_N| + \sum_{j=1}^N|(S_N\cdots S_j)(u_{j-1})| + |\rho(c)| + \sum_{j=1}^N|(S_N\cdots S_j)(u_{j-1}^{-1})| + |u_N^{-1}|\\
&\leq 2|\rho(ab)|+|\rho(c)|\\
&\leq (N+2)\left(|\rho(a)|+|\rho(b)|\right),
\end{align*}
since
\[\rho(ab)=abu_NS_N(u_{N-1})(S_NS_{N-1})(u_{N-2})\cdots (S_N\cdots S_1)(u_0)\]
and $|\rho(ab)|=|\rho(a)|+|\rho(b)|$.
\end{proof}
Note that the proof above also shows that if $C$ is the set of cancellations between $x,y$, then for any cancellation $q$ of $\rho(x),\rho(y)$ we have
\[|q|\leq \max\set{\ |\rho(c)|\ \mid c\in C} + 2(|\rho(a)|+|\rho(b)|).\]
Now we have a tool to bound the length of a cancellation between two words in $B_0$ from above. The next lemma bounds the length of a word in $B_0$ from below. We will in the following call $x\in \F_2$  \textbf{positive} if $x$ consists only  of positive powers of $a,b$. Similarly, we say $x$ is \textbf{negative} if $x$ consists only  of negative powers of $a,b$. Note that if $x$ is either positive or negative then we clearly have $|x|=\|x\|$.
\begin{lem}\label{estimatewordlength}
Let $\rho\in \FR$ and $z\in \set{w_0,v_0,\ldots, w_7,v_7}$. Then 
\[\|\rho(z)\|\geq \left(\frac{n(n-1)}{2}-2n\right)\left(|\rho(a)|+|\rho(b)|\right).\] 

\end{lem}

\begin{proof} It is enough to consider $w_0,\ldots, w_7$.
Moreover, it is clear that if $z\in\set{w_0,w_1, 
w_4,w_5}$, then 
\[\|\rho(z)\|=\frac{n(n-1)}{2}\left(|\rho(a)|+|\rho(b)|\right),\]
since there is no cancellation  due to the fact that $z$ is either positive or negative.  

For the remaining cases, we will begin with some observations. Assume $p,p_+,p_-\in \F_2$ satisfy that $p=p_+p_-$ is reduced, $p_+$ is positive and $p_-$ is negative. Moreover, let $S\in \set{\phi,\psi}$ and $u\in \set{a,b}$ be such that $u=a\iff S=\psi$. Then we have $S(p)=S(p_+)S(p_-)$. If both $p_+$ and $p_-$ are non-trivial, $S(p_+)$ will end with $u$ and $S(p_-)$ will begin with $u^{-1}$. Thus $uu^{-1}$ will be removed in the product. Since $p_+p_-$ is reduced there will not be any other reduction in $S(p_+)S(p_-)$. Note also that $S(p_+)$ is positive and  $S(p_-)$ is negative. Lastly, note that if instead $p=p_-p_+$ is reduced, then $S(p)=S(p_-)S(p_+)$ is also reduced.

Now let $x\in \F_2$ be neither positive or negative and let $\rho\in \FR$. Fix $S_0\ldots ,S_N\in \set{\phi,\psi}$ and $u_0,\ldots,u_N\in \set{a,b}$ such that $\rho=S_N\cdots S_0$ and $u_i=a\iff S_i=\psi$, for all $0\leq i\leq N$. 

Assume first that for all $0\leq j<N$ we have $S_j\cdots S_0(x)$ is neither positive nor negative.
 Then for each $0\leq i\leq N$, we may fix $k_i\geq 1$ together with positive  $p_{(i,+)}^1,\ldots p_{(i,+)}^{k_i}\in \F_2\setminus \set{e}$, and negative  $p_{(i,-)}^1,\ldots,p_{(i,-)}^{k_i}\in \F_2\setminus \set{e}$, such that for each $0\leq i<N$ there is a cyclically reduced cyclic conjugate of $x$ and of $S_{i}\cdots S_0(x)$ of the form
 \[p_{(0,+)}^1p_{(0,-)}^1p_{(0,+)}^2p_{(0,-)}^2\cdots p_{(0,+)}^{k_0}p_{(0,-)}^{k_0}\]
 and
\[p_{(i+1,+)}^1p_{(i+1,-)}^1p_{(i+1,+)}^2p_{(i+1,-)}^2\cdots p_{(i+1,+)}^{k_{i+1}}p_{(i+1,-)}^{k_{i+1}},\]
respectively.
Then $k_0\geq k_1\geq \cdots \geq k_{N}$ and hence, by the observations above, we have
\begin{align*}
\|\rho(x)\|&=|\rho(\ol{x})|-2k_N|u_N|-\sum_{i=0}^{N-1}2k_i|S_N\cdots S_{i+1}(u_i)|\\
&\geq |\rho(\ol{x})|-2k_0\left(|\rho(a)|+|\rho(b)|\right).
\end{align*}

Next, assume that  $0\leq j <N$ is  least such that $S_j\cdots S_0(x)$ is either positive or negative. Then, as before, we may for each $0\leq i\leq j$ chose $l_i\geq 1$ together with positive  $q_{(i,+)}^1,\ldots q_{(i,+)}^{l_i}\in \F_2\setminus \set{e}$ and negative  $q_{(i,-)}^1,\ldots,q_{(i,-)}^{l_i}\in \F_2\setminus \set{e}$ such that for each $0\leq i<j$ there is a cyclically reduced cyclic conjugate of $x$ and of $S_{i}\cdots S_0(x)$ of the form
 \[q_{(0,+)}^1q_{(0,-)}^1q_{(0,+)}^2q_{(0,-)}^2\cdots q_{(0,+)}^{l_0}q_{(0,-)}^{l_0}\]
 and
\[q_{(i+1,+)}^1q_{(i+1,-)}^1q_{(i+1,+)}^2q_{(i+1,-)}^2\cdots q_{(i+1,+)}^{l_{i+1}}q_{(i+1,-)}^{l_{i+1}},\]
respectively.
Then $l_0\geq l_1\geq \cdots \geq l_{j}$ and hence, by the observations above, we have
\begin{align*}
\|\rho(x)\|&=|\rho(\ol{x})|-\sum_{i=0}^{j}2l_i|S_N\cdots S_{i+1}(u_i)|\\
&\geq |\rho(\ol{x})|-2l_0\left(|\rho(a)|+|\rho(b)|\right).
\end{align*}

Finally, to finish the proof, note that for any $z\in \set{w_2,w_3,w_6,w_7}$, we can chose $k_0,l_0=n$ in the argument above. Thus, as $$|\rho(\ol{z})|=\frac{n(n-1)}{2}\left(|\rho(a)|+|\rho(b)|\right),$$
we obtain
\begin{align*}
\|\rho(z)\|\geq \left(\frac{n(n-1)}{2}-2n\right)\left(|\rho(a)|+|\rho(b)|\right),
\end{align*}
as wanted.
\end{proof}

Note that 
\[8n<\frac{1}{6}\left(\frac{n(n-1)}{2}-2n\right),\]
since $n>101$. This will be used all the time below to conclude that there is no bad cancellation in the various cases.

We will now begin to argue that there is no bad cancellation between two words from $B_0$. So let $x,y\in \set{w_0,v_0,\ldots,w_7,v_7}$. The following decomposition will be useful. For $m\in \set{1,\ldots,n}$ let
\[w_0^m=a^mb^m, \qquad w_1^m=a^{-m}b^{-m}, \qquad w_3^m=a^mb^{-m},\] \[w_4^m=b^ma^m, \qquad w_5^m=b^{-m}a^{-m}, \qquad w_7^m=b^ma^{-m}\]
and for $m\in \set{1,\ldots, n-1}$ let
\[w_2^m=b^ma^{-m-1},\qquad w_2^n=b^na^{-1}, \qquad w_6^m=a^mb^{-m-1}, \qquad w_6^n=a^nb^{-1}.\]
Then for $i\in \set{0,\ldots, 7}$ we have, up to cyclic permutation, that
\[w_i=w_i^1w_i^2\cdots w_i^n\]
and 
\[\rho(w_i^1)\rho(w_i^2)\cdots \rho(w_i^n)\]
is a reduced word whenever each factor is reduced, for all $\rho\in \FR$. However, it is not necessarily cyclically reduced. If for some $k\geq 1$ we have $\rho=\phi^k$ or $\rho=\psi^k$ or $i\in \set{0,1,3,4,5,7}$, then 
\[\rho(w_i^1)\rho(w_i^2)\cdots \rho(w_i^n)\]
is cyclically reduced. If $i\in \set{2,6}$ and $\rho\notin \set{\phi^k,\psi^k\mid k\in \N}$, then any possible reduction in the induced cycle of $\rho(w_i^1)\rho(w_i^2)\cdots \rho(w_i^n)$ is contained in $\rho(w_i^n)\rho(w_i^1)$.
\\
\\
\textbf{Claim 1:} If $\rho\in \FR$ and $x\neq y$, then there is no bad cancellation between $\rho(x)$ and $\rho(y)$.
{\begin{proof}
It  is easy to check that $N=2n-2$ satisfies the assumption of Lemma \ref{estimatecancellation}, since $x\neq y$. Thus any cancellation, $q$, between $\rho(x)$ and $\rho(y)$ satisfies 
\[|q|\leq 2n\left(|\rho(a)|+|\rho(b)|\right).\]
So, by Lemma \ref{estimatewordlength}, there cannot be any bad cancellation between $\rho(x)$ and $\rho(y)$.
\end{proof} 

In the following, we let $$A_\phi=\set{ \phi\rho\mid \rho\in \FR} \ \text{and} \ A_\psi=\set{\psi\rho \mid \rho\in \FR}.$$ Note that $\FR\setminus \set{1}=A_\phi\sqcup A_\psi$.
\\
\\
\textbf{Claim 2:} If $\rho,\sigma, \eta\in \FR$ with $\rho=\eta\sigma$ and $\sigma\neq 1$, then there is no bad cancellation between $\rho(x)$ and $\eta(y)$.
\begin{proof} Note that either $\sigma\in A_\phi$ or $\sigma\in A_\psi$. Assume without loss of generality that we are in the first case. Then the only powers of $a$ occurring in $\sigma(x)$ are $a^1$ and $a^{-1}$. Thus any cancellation between $\sigma(x)$ and $y$ is a substring of the cycle induced by $y$, which only contains these powers. Therefore it is easily seen that $N=n+1$ satisfies the assumptions of \cref{estimatecancellation}. So any cancellation, $q$, between $\rho(x)$ and $\eta(y)$ satisfies
\[|q|\leq (n+3)\left(|\eta(a)|+|\eta(b)|\right).\]
Moreover, by  Lemma \ref{estimatewordlength}, it holds that
\[\|\rho(x)\|,\|\eta(y)\|\geq \left(\frac{n(n-1)}{2}-2n\right)(|\eta(a)|+|\eta(b)|),\]
since $|\rho(a)|+|\rho(b)|\geq |\eta(a)|+|\eta(b)|$. Thus there is no bad cancellation between $\rho(x)$ and $\eta(y)$.
\end{proof}

Now we will take care of the case where $\rho_1,\rho_2\in \FR$ are different, but none of them extends the other.
\\
\\
\textbf{Claim 3:} If $\rho_1,\rho_2,\sigma_1,\sigma_2,\eta_1,\eta_2\in \FR$ with $\sigma_1\in A_\phi$, $\sigma_2\in A_\psi$, $$\rho_1=\eta_1\sigma_1\qquad \text{and} \qquad \rho_2=\eta_2\sigma_2,$$ then there is no bad cancellation between $\rho_1(x)$ and $\rho_2(y)$.
\begin{proof} First note that $\sigma_1(x)$ will only contain $a,a^{-1}$ as powers of $a$, while  $\sigma_2(y)$ will only contain $b,b^{-1}$ as powers of $b$. 

We now claim that for each $i\in \set{0,\ldots,7}$ and $m\in \set{3,\ldots, n-1}$ we have $\sigma_1(w_i^m)$ contains the string $b^l$ or $b^{-l}$, for some $l\geq 2$. Indeed $\sigma_1$ is of one of the forms $\phi^k$, $\sigma_1^0\psi\phi^k$, $\phi\psi^k$ or $\sigma_1^0\phi\psi^k$, for some $k\geq 1$ and $\sigma_1^0\in A_\phi$. By straightforward calculation the statement is clearly true for $\sigma_1=\phi^k$ or $\sigma_1=\phi\psi^k$. To see that the statement also holds in the remaining cases, one may consider $\phi\psi^k(w_i^m)$ and $\psi\phi^k(w_i^m)$ and then use the fact that $\sigma_1^0(a)=aub$ and $\sigma_1^0(b)=bzb$ or $\sigma_1^0(b)=b$, for some positive $u,z\in \F_2$. 

Similarly, for each $i\in \set{0,\ldots,7}$ and $m\in \set{3,\ldots, n-1}$, we have $\sigma_2(w_i^m)$ contains the string $a^l$ or $a^{-l}$ for some $l\geq 2$.

Now let $i,j\in \set{0,\ldots,7}$ satisfy that $x\in \set{w_i,v_i}$ and $y\in\set{w_j,v_j}$. Then from the above it follows that any cancellation, $q$, between $\sigma_1(x)$ and $\sigma_2(y)$ is contained in either
\[\sigma_1(w_i^{n-1})\sigma_1(w_i^n)\sigma_1(w_i^1)\sigma_1(w_i^2)\sigma_1(w_i^3),\] 
\[\sigma_1(w_i^m)\sigma_1(w_i^{m+1}),\]
or in one of their inverses, for some $m\in \set{3,\ldots, n-2}$.
Therefore, by Lemma \ref{estimatecancellation}, we have 
\begin{align*}
|q|&\leq 3n(|\rho_1(a)|+|\rho_1(b)|).
\end{align*}
So, by Lemma \ref{estimatewordlength}, we have $|q|<\frac{1}{6}\|\rho_1(x)\|$.

Similarly any cancellation, $q$, between $\sigma_1(x)$ and $\sigma_2(y)$ is contained in either
\[\sigma_2(w_j^{n-1})\sigma_2(w_j^n)\sigma_2(w_j^1)\sigma_2(w_j^2)\sigma_2(w_j^3),\] 
\[\sigma_2(w_j^m)\sigma_2(w_j^{m+1})\]
or in one of their inverses, for some $m\in \set{3,\ldots, n-2}$. So, by Lemma \ref{estimatecancellation} and  Lemma \ref{estimatewordlength}, we also have $|q|<\frac{1}{6}\|\rho_2(y)\|$. Thus
 there cannot be any bad cancellation between $\rho_1(x)$ and $\rho_2(y)$.
 \end{proof}

From Claim 1, Claim 2 and Claim 3 we may conclude that there is no bad cancellation between two words in $B_0$, i.e., that $B_0$ satisfies the $C'(1/6)$ cancellation property. 

We will now prove that there is no bad cancellation between a word from $B_0$ and a word from $B_1$. To do so, let $A_\phi^0=\set{\phi^k\mid k\geq 1}$, $A_\psi^0=\set{\psi^k\mid k\geq 1}$ and
\[A^1=\set{ \eta\phi\psi^k ,\eta\psi\phi^k\mid k\geq 1, \eta\in \FR}.\]
Then $\FR\setminus \set{1}=A_\phi^0\sqcup A_\psi^0\sqcup A^1$. We will again consider fixed $x,y\in \set{w_0,v_0,\ldots,w_7,v_7}$. Let $i,j\in \set{0,\ldots, 7}$ be fixed such that $x\in \set{w_i,v_i}$ and $y\in \set{w_j,v_j}$.
\\
\\
\textbf{Claim 4:} If $\rho,\sigma\in \FR$ with $\rho=1$ and $\sigma\neq 1$, then there is no bad cancellation between $\rho(x)$ and $\xi\sigma(y)$.
\begin{proof}
This follows by the same arguments as the ones used in the beginning of the proof of Claim 2.
\end{proof}


\noindent
\textbf{Claim 5:} If $\rho\in A_\phi$, $\sigma\in A_\psi$ or $\rho\in A_\psi$, $\sigma\in A_\phi$. Then there is no bad cancellation between $\rho(x)$ and $\xi\sigma(y)$.
\begin{proof}
This follows by arguments similar to those in the beginning of the proof of Claim 3.
\end{proof}


From Claim 4, we may assume that both $\rho,\psi\in A_\phi^0\sqcup A_\psi^0\sqcup A^1$. Moreover, by Claim 5, there is no bad cancellation between $\rho(x)$ and $\xi\sigma(y)$ in the case $\rho\in A_\phi^0$ and $\sigma\in A_\psi^0$ or in the case $\rho\in A_\psi^0$ and $\sigma\in A_\phi^0$. In the next three claims, we prove that there is no bad cancellation within each of these sets.\\
\\
\textbf{Claim 6:} If $k,l\geq 1$, then there is no bad cancellation between $\phi^k(x)$ and $\xi\phi^l(y)$.

\begin{proof} 
Consider $\phi^t(u),\xi\phi^t(z)$ for $u,z\in \set{w_0,v_0,\ldots,w_7,v_7}$ and $t\geq 1$. Either all the powers of  $a$ are positive or all the powers of $a$ are negative. Below we have put these observations into a table. Here  $+$ and $-$ refer to the sign of the occurring powers of $a$. 

\begin{center}
\begin{tabular}{l||l|l|l|l|l|l|l|l}
 & $w_0$  & $w_1$  & $w_2$ & $w_3$ & $w_4$ & $w_5$ & $w_6$ & $w_7$   \\ \hline \hline 
$\phi^k$ &  $+$  &  $-$ &  $-$ &  $+$ &  $+$  &  $-$ &  $+$ &  $-$  \\ \hline
$\xi\phi^l$ &  $-$  & $+$ &  $+$ &  $-$ &   $-$  &  $+$ &  $-$ &  $+$
\end{tabular}
\end{center}
\begin{center}
\begin{tabular}{l||l|l|l|l|l|l|l|l}
 & $v_0$  & $v_1$  & $v_2$ & $v_3$ & $v_4$ & $v_5$ & $v_6$ & $v_7$   \\ \hline \hline 
$\phi^k$ &  $-$  &  $+$ & $+$ &  $-$ & $-$  & $+$ &  $-$ &  $+$  \\ \hline
$\xi\phi^l$ &$+$  &  $-$ &  $-$ & $+$ &  $+$  & $-$ &  $+$ &  $-$
\end{tabular}
\end{center}
It is easily seen that if the signs of the powers of $a$ do not match, then there is no bad cancellation between $\phi^k(z)$ and $\xi\phi^l(u)$ for the corresponding $u,z\in \set{w_0,v_0,\ldots,w_7,v_7}$. So assume that the sign of the powers of $a$ in $\phi^k(x)$ and the sign of the powers of $a$ in $\xi\phi^l(y)$ match. Then the signs of the powers of $a$ in $x$ and $y$ do not match. Assume without loss of generality that the powers of $a$ in $x$ are negative and the powers of $a$ in $y$ are positive. Then, if $x=w_i$, the string $(\phi^k(a^{-1}))^2=(b^{-k}a^{-1})^2$ is contained in $\phi^k(w_i^m)$, for  $m\in \set{3,\ldots , n-1}$. If $x=v_i$, the string $(b^{-k}a^{-1})^2$ is contained in  $\phi^k(w_i^m)^{-1}$, for all $m\in \set{3,\ldots , n-1}$.
Similarly, if $y=w_j$, the string $(\xi\phi^l(a))^2=(a^{-1}b^{l})^2$ is contained in  $\xi\phi^l(w_j^m)$, for  $m\in \set{3,\ldots , n-1}$. If $y=v_j$, the string $(a^{-1}b^{l})^2$ is contained in  $\xi\phi^l(w_j^m)^{-1}$, for all $m\in \set{3,\ldots , n-1}$. Moreover, the string $(b^{-k}a^{-1})^2$ does not appear in $\xi\phi^l(y)$ and the string $(a^{-1}b^l)^2$ does not appear in $\phi^k(x)$. Therefore any cancellation, $q$, between $\phi^k(x)$ and $\xi\phi^l(y)$ is contained in either
\[\phi^k(w_i^{n-1})\phi^k(w_i^n)\phi^k(w_i^1)\phi^k(w_i^2)\phi^k(w_i^3),\] 
\[\phi^k(w_i^m)\phi^k(w_i^{m+1})\]
or in one of their inverses, for some $m\in \set{3,\ldots, n-2}$. Similarly, $q$ is also contained in either
\[\xi\phi^l(w_j^{n-1})\xi\phi^l(w_j^n)\xi\phi^l(w_j^1)\xi\phi^l(w_j^2)\xi\phi^l(w_j^3),\] 
\[\xi\phi^l(w_j^m)\xi\phi^l(w_j^{m+1})\]
or in one of their inverses for some $m\in \set{3,\ldots, n-2}$.
Thus, by \cref{estimatecancellation}, we have 
\[|q|\leq 3n\min\set{ |\phi^k(a)|+|\phi^k(b)|, |\xi\phi^l(a)|+|\xi\phi^l(b)|},\]
and hence, by \cref{estimatewordlength}, there is no bad cancellation between $\phi^k(x)$ and $\xi\phi^l(y)$. 
\end{proof}

\noindent
\textbf{Claim 7:} If $k,l\geq 1$, then there is no bad cancellation between $\psi^k(x)$ and $\xi\psi^l(y)$. 
\begin{proof}
First, consider the sign of the powers of $b$ occurring in $\psi^k(u)$ and $\xi\psi^l(z)$ for $u,z\in \set{w_0,v_0,\ldots,w_7,v_7}$.
 \begin{center}
\begin{tabular}{l||l|l|l|l|l|l|l|l}
 & $w_0$  & $w_1$  & $w_2$ & $w_3$ & $w_4$ & $w_5$ & $w_6$ & $w_7$   \\ \hline \hline 
$\psi^k$ &  $+$  &  $-$ &  $+$ &  $-$ &  $+$  &  $-$ &  $-$ &  $+$  \\ \hline
$\xi\psi^l$ &  $+$  & $-$ &  $+$ &  $-$ &   $+$  &  $-$ &  $-$ &  $+$
\end{tabular}
\end{center}
\begin{center}
\begin{tabular}{l||l|l|l|l|l|l|l|l}
 & $v_0$  & $v_1$  & $v_2$ & $v_3$ & $v_4$ & $v_5$ & $v_6$ & $v_7$   \\ \hline \hline 
$\psi^k$ &  $-$  &  $+$ & $-$ &  $+$ & $-$  & $+$ &  $+$ &  $-$  \\ \hline
$\xi\psi^l$ &$-$  &  $+$ &  $-$ & $+$ &  $-$  & $+$ &  $+$ &  $-$
\end{tabular}
\end{center}
It is easily seen that if the sign in two cells does not match, then there is no bad cancellation between  $\psi^k(u)$ and $\xi\psi^l(z)$,  for the corresponding $u,z\in \set{w_0,v_0,\ldots, w_7,v_7}$.
In case the signs of the powers of $b$ are the same, one may use a similar argument as the one in Claim 7. Assume first that the signs of $b$ are both positive. Then, if $x=w_i$, the string $(ba^k)^2$ is contained in  $\psi^k(w_i^m)$, for  $m\in \set{3,\ldots , n-1}$. If $x=v_i$, the string $(ba^k)^2$ is contained in  $\psi^k(w_i^m)^{-1}$, for all $m\in \set{3,\ldots , n-1}$.
Similarly, if $y=w_j$, the string $(ba^{-l})^2$ is contained in  $\xi\psi^l(w_j^m)$, for  $m\in \set{3,\ldots , n-1}$. If $y=v_j$, the string $(ba^{-l})^2$ is contained in  $\xi\psi^l(w_j^m)^{-1}$, for all $m\in \set{3,\ldots , n-1}$.
Moreover, the string $(ba^k)^2$ will not be contained in $\xi\psi^l(y)$, while the string $(ba^{-l})^2$ will not be contained in $\psi^k(x)$. Thus one may deduce, as in the proof of Claim 6, that there cannot be any bad cancellation between $\psi^k(x)$ and $\xi\psi^l(y)$ in this case. In case the powers of $b$ are both negative a similar argument will work.
\end{proof}


\noindent
\textbf{Claim 8:} If $\rho,\sigma\in A^1$, then there is no bad cancellation between $\rho(x)$ and $\xi\sigma(y)$.
\begin{proof}
First, by considering the form of $\psi\phi^k(z)$ and $\phi\psi^k(z)$ for the words $z\in \set{w_0,v_0,\ldots,w_7,v_7}$, one finds that for all $\eta\in \FR$, we have $\eta\psi\phi^k(z)$ and $\eta\phi\psi^k(z)$ will contain at most one occurrence of one of the strings
\[ab^{-t}a, \qquad a^{-1}b^ta^{-1}, \qquad ba^{-t}b \quad \text{or} \quad b^{-1}a^tb^{-1},\]
for some $t\geq 1$. Moreover, for all $r\in \set{0,\ldots,7}$ and $m\in \set{2,\ldots, n-1}$, we have that $\eta\phi\psi
^k(w_r^m)$ and $\eta\psi\phi^k(w_r^m)$ contain at least one of the strings
\[ab^{t}a, \qquad a^{-1}b^{-t}a^{-1}, \qquad ba^{t}b \quad \text{or} \quad b^{-1}a^{-t}b^{-1},\]
for some $t\geq 1$.
This is again  straightforward to check, by considering the form of $\psi\phi^k(w_r^m)$ and $\phi\psi^k(w_r^m)$.

Note that the above implies that $\rho(x)$ contains at most one of the strings \[ab^{-t}a, \qquad a^{-1}b^ta^{-1}, \qquad ba^{-t}b \quad \text{or} \quad b^{-1}a^tb^{-1},\]
for some $t\geq 1$.  Moreover, for all $m\in \set{2,\ldots, n-1}$, we have that $\rho(w_i^m)$ contains at least one of the strings \[ab^{t}a, \qquad a^{-1}b^{-t}a^{-1}, \qquad ba^{t}b \quad \text{or} \quad b^{-1}a^{-t}b^{-1},\]
for some $t\geq 1$.

Conversely, the above also implies that $\xi\sigma(y)$ contains at most one of the strings  \[ab^{t}a, \qquad a^{-1}b^{-t}a^{-1}, \qquad ba^{t}b \quad \text{or} \quad b^{-1}a^{-t}b^{-1},\]
for some $t\geq 1$. Moreover, for all $m\in \set{2,\ldots, n-1}$, we have that $\xi\sigma(w_j^m)$ contains at least one of the strings \[ab^{-t}a, \qquad a^{-1}b^ta^{-1}, \qquad ba^{-t}b \quad \text{or} \quad b^{-1}a^tb^{-1},\]
for some $t\geq 1$.

Thus, by making considerations and use of \cref{estimatecancellation} and \cref{estimatewordlength}, as in the proof of the earlier claims, we may conclude that there cannot be any bad cancellation between $\rho(x)$ and $\xi\sigma(y)$.
\end{proof}

\noindent
\textbf{Claim 9:} If $\rho\in A^1$ and $\sigma\in A_\phi^0\cup A_\psi^0$, then there is no bad cancellation between $\rho(x)$ and $\xi\sigma(y)$.
\begin{proof}
From earlier results it is enough to consider the case where $\rho,\sigma\in A_\phi$ or $\rho,\sigma\in A_\psi$. Assume without loss of generality that we are in the first case. Then there exist $l,k\geq 1$ such that $\sigma=\phi^k$ and  $\rho=\phi^l\eta$, for some $\eta\in A_\psi$. Then the only possible powers of $b$ occurring in $\rho(x)$ are $b,b^{-1},b^l,b^{-l},b^{l+1},b^{-l-1}$.
Thus, if $q$ is a cancellation between $\rho(x)$ and $\sigma(y)$, then there is $m\in \set{1,n-7}$ or $t\in \set{0,\ldots, 6}$ such that $q$ is contained in either
\[\sigma(w_j^m)\sigma(w_j^{m+1})\cdots \sigma(w_j^{m+7}),\]
\[\sigma(w_j^{n-t}\sigma(w_j^{n-t+1})\cdots \sigma(w_j^n)\sigma(w_j^1)\sigma(w_j^2)\cdots \sigma(w_j^{8-t-1})\]
or in one of their inverses.

Therefore, by \cref{estimatecancellation} we obtain
\[|q|\leq 8n(|\sigma(a)|+|\sigma(b)|)\]
and hence, by \cref{estimatewordlength}, we have  $|q|<\frac{1}{6}\|\sigma(y)\|$. 
If $l\geq k$, then by \cref{estimatewordlength} and the fact that $|\rho(a)|+|\rho(b)|\geq |\sigma(a)|+|\sigma(b)|$, we also have
\begin{align*}
\|\rho(x)\|&\geq \left(\frac{n(n-1)}{2}-2n\right)(|\rho(a)|+|\rho(b)|)\\
&\geq \left(\frac{n(n-1)}{2}-2n\right)(|\sigma(a)|+|\sigma(b)|)\\
&> 6|q|.
\end{align*}
So assume $l<k$. Then we have that each of $b^l$ and $b^{-l}$ occurs at most once in $\sigma(y)$. But by the arguments in Claim 3, we have that for all $m\in \set{3,\ldots, n-1}$ there is $t\geq 2$ such that $a^t$ or $a^{-t}$ occur in $\eta(w_i^m)$. Thus $b^l$ or $b^{-l}$ occur in $\rho(w_i^m)$ for all $m\in \set{3,\ldots, n-1}$. 

Therefore $q$ is contained in either
\[\rho(w_i^{n-3})\rho(w_i^{n-2})\rho(w_i^{n-1})\rho(w_i^n)\rho(w_i^1)\rho(w_i^2)\rho(w_i^3)\rho(w_i^4)\rho(w_i^5),\]
\[\rho(w_i^{m-1})\rho(w_i^m)\rho(w_i^{m+1})\rho(w_i^{m+1})\]
or in one of their inverses, for some $m\in \set{4,\ldots, n-3}$. Thus we obtain
\[|q|\leq 5n\left(|\rho(a)|+|\rho(b)|\right).\]
So, by \cref{estimatewordlength}, we have $|q|<\frac{1}{6}\|\rho(x)\|$, as well.
We may therefore conclude that there is no bad cancellation between $\rho(x)$ and $\xi\sigma(y)$.
\end{proof}


Putting all the claims together we finally have a proof of the following theorem.
\begin{thm}\label{goal}
Let $w=aba^2b^2\cdots a^nb^n$ for some $n>101$. Then there is a transversal $T$ for the left cosets in $\Aut(\F_2)/\F_2$ such that the set
\[\set{\eta(w)\mid \eta\in T}\]
satisfies the $C'(1/6)$ cancellation property.
\end{thm}
Using this, and the criterion in \cref{impli}, it is now easy to construct continuum many non-atomic invariant random subgroups of $\F_2$, which are $\Aut(\F_2)$-invariant and weakly mixing with respect to the action of $\Aut(\F_2)$.

Fix by \cref{goal} a transversal $T$ for the left cosets in $\Aut(\F_2)/\F_2$ such that the set $\set{\eta(w)\mid \eta\in T}$ satisfies the $C'(1/6)$ cancellation property. Moreover, fix an enumeration $T=\set{\eta_i\mid i\in \N}$ such that $\eta_0$ is the identity. Then we have
\[\eta_0(w)\notin \langle\langle \eta_i(w)\mid i\geq 1\rangle\rangle_{\F_2} ,\]
since it follows by \cref{estimatewordlength} that $\|\eta_i(w)\|\geq |\eta_0(w)|$ for all $i\geq 1$. Thus, by \cref{inverstrans}, condition $(1)$ in \cref{impli} is satisfied, and hence the proof of $(1)\implies (2)$ in \cref{impli} provides a family as the one described above.

There is also another consequence of \cref{goal}.
Let $C$ denote the set  of conjugacy classes of $\F_2$ and let $G=\Out(\F_2)= \Aut(\F_2)/\F_2$ be the outer automorphism group of $\F_2$. Now consider the natural action $G\ac^a C$. It is well known that there is a conjugacy class $c\in C$ such that $\stab_a(c)=\set{e}$  (see \cite[page 45]{LS77}). Therefore for such $c\in C$ we have
\[c\cap \bigcup\{g\cdot^a c\mid g\in G\setminus\{e\}\} = \emptyset.\]
With \cref{goal} we obtain the following strengthening of this result.

\begin{cor}\label{last} There exists a conjugacy class $c\in C$ such that $$c\cap \langle g\cdot^a c\mid g\in G\setminus\{e\}\rangle_{\F_2} = \emptyset,$$
that is, $c$ is disjoint from the (normal) subgroup generated by the conjugacy classes $g\cdot^a c$ for  $g\neq e$.
\end{cor}

\begin{proof}
Let $c$ be the conjugacy class of $w=aba^2b^2\cdots a^nb^n$ for some $n>101$ and assume towards contradiction that
\[c\cap \langle g \cdot^a c \mid g\in G\setminus \set{e}\rangle_{\F_2}\neq \emptyset.\]
Now choose by \cref{goal} a transversal $T$ for the left cosets in $\Aut(\F_2)/\F_2$ such that $\set{\eta(w)\mid \eta\in T}$ satisfies the $C'(1/6)$ cancellation property. Fix $t_0\in T$ such that $t_0\in \F_2$. Then we have
\[w\in \langle g \cdot^a c \mid g\in G\setminus \set{e}\rangle_{\F_2} =\langle\langle \eta(w)\mid \eta\in T\setminus \set{t_0}\rangle\rangle_{\F_2},\]
which contradicts that $\|\eta(w)\|\geq |w|$ for all $\eta\in T\setminus \set{t_0}$.
\end{proof}

We do not know if the analogs of \cref{goal} and \cref{last} hold for any $\F_n$ with $n>2$.

\noindent Department of Mathematics

\noindent California Institute of Technology

\noindent Pasadena, CA 91125

\noindent\textsf{kechris@caltech.edu }

\bigskip

\noindent Department of Mathematical Sciences

\noindent University of Copenhagen

\noindent Universitetsparken 5

\noindent DK-2100 Copenhagen

\noindent\textsf{vibquo@math.ku.dk}


\begin{thebibliography}{MMMMM}


\bibitem[AE11]{AE11} M. Ab{\'e}rt and G. Elek,  The space of actions, partition metric and combinatorial rigidity,  {\it  arXiv:1108.2147v1}.

\bibitem[AGV14]{AGV14} M. Ab{\'e}rt, Y. Glasner and B. Vir{\'a},  Kersten's theorem for invariant random subgroups ,  {\it Duke Math. J.}, {\bf 163(3)} (2014), 465--488.

\bibitem[A79]{A79} S.I. Adian, {\it The Burnside Problem and Identities in Groups}, Springer-Verlag, (1979).


\bibitem[BDLW16]{BDLW16} U. Bader, B. Duchesne, J. L\'{e}cureux and P. Wesolek,  Amenable invariant random subgroups ,  {\it Israel J. Math.}, {\bf 213(1)} (2016), 399--422.


\bibitem[BLT18]{BLT18} O. Becker, A. Lubotzky and A. Thom, Stability and invariant random subgroups, {\it arXiv: 1801.08381v2}.


\bibitem[BT18]{BT18} F. Bencs and L. M\'{a}rton T\'{o}th, Invariant random subgroups acting on trees, {\it arXiv: 1801.05801v1}.

\bibitem[BGN15]{BGN15} M.G. Benli, R. Grigorchuk and T. Nagnibeda, Universal groups of intermediate growth and their invariant random subgroups, {\it Funct. Anal. and Appl.}, {\bf 49(3)} (2015), 159--174. 

\bibitem[B18]{B18} A. Bernshteyn, Multiplication of weak equivalence classes may be discontinuous,  {\it arXiv: 1803.09307v1}. 


\bibitem[BBT17]{BBT17} I. Biringer, L. Bowen and O. Tamuz, Invariant random subgroups of semidirect products, {\it arXiv: 1703.01282v1}.

\bibitem[BT17]{BT17} I. Biringer and O. Tamuz, Unimodularity of invariant random subgroups,  {\it Trans. Amer. Math. Soc.}, {\bf 369(6)} (2017), 4043--4061.

 \bibitem[B08]{B08} O. Bogopolski,  {\it Introduction to Group Theory},  Europ. Math. Soc.,  (2008). 

\bibitem[Bo14]{Bo14} L. Bowen, Random walks on random coset spaces with applications to Furstenberg entropy, {\it Invent. Math.}, {\bf 196(2)} (2014), 485--510.

\bibitem[Bo15]{Bo15} L. Bowen, Invariant random subgroups  of the free group, {\it Groups Geom. Dyn.}, {\bf 19(3)} (2015), 891--916.

\bibitem[B16]{B16} P.J. Burton,  Topology and convexity in the space of actions modulo weak equivalence,  {\it arXiv:1501.04079v2}.

\bibitem[BK18]{BK18} P.J. Burton and A.S. Kechris, Weak containment of measure preserving group actions, preprint, (2018).


\bibitem[BGK15]{BGK15} L. Bowen, R. Grigorchuk and R. Kravchenko, Invariant random subgroups of lamplighter groups,  {\it Israel J. Math.}, {\bf207(2)} (2015), 763--782. 


\bibitem[BGK17]{BGK17} L. Bowen, R. Grigorchuk and R. Kravchenko, Characteristic random subgroups of geometric groups and free abelian groups of infinite rank,  {\it Trans. Amer. Math. Soc.}, {\bf 369} (2017), 755--781. 

\bibitem[CMZ81]{CMZ81} M. Cohen, W. Metzler and A. Zimmermann, What Does a Basis of $F(a,b)$ Look Like?,  {\it Math. Ann.}, {\bf 257} (1981), 435--445. 

\bibitem[DM17]{DM17} A. Dudko and K. Medynets, On invariant random subgroups of block-diagonal limits of symmetric groups, {\it arXiv: 1711.01653v1}.


 \bibitem[EG16]{EG16} A. Eisenmann and Y. Glasner, Generic invariant random subgroups, after Bowen, {\it Proc. Amer. Math. Soc.}, {\bf 144(10)} (2016), 4231--4246. 

\bibitem[E12]{E12} G. Elek, Finite graphs and amenability,  {\it J. Funct. Anal.}, {\bf 263} (2012), 2593--2614. 


\bibitem[GeL18]{GeL18} I. Gekhtman and A. Levit, Critical exponents of invariant random subgroups in negative curvature, {\it arXiv: 1804.02995v3}.



\bibitem[Ge15]{Ge15} T. Gelander, Lecture notes on invariant random subgroups and lattices in rank one and higher rank, {\it arXiv: 1503.08402v2}.

\bibitem[Ge18]{Ge18} T. Gelander, The Kazhdan-Margulis theorem for invariant random subgroups, {\it Adv. Math.}, {\bf 327} (2018), 47--51.

\bibitem[GL16]{GL16} T. Gelander and A. Levit, Invariant random subgroups over non-archimedean local fields, {\it arXiv: 1707.03578v2}.

\bibitem[G03]{G03} E. Glasner, {\it Ergodic Theory via Joinings}, Amer. Math. Soc., (2003).

 \bibitem[G17]{G17} Y. Glasner, Invariant random subgroups of linear groups,   {\it Israel J. Math.}, {\bf 219} (2017), 215--270. 
 
  \bibitem[HT16]{HT16} Y. Hartman and O. Tamuz, Stabilizer rigidity in irreducible group actions  {\it Israel J. Math.}, {\bf 216(2)} (2016), 679--705. 
 
  \bibitem[HY17]{HY17} Y. Hartman and A.Yadin, Furstenberg entropy of intersectional invariant random subgroups, {\it arXiv: 1701.08350v1}. 
 
 
\bibitem[I11]{I11} A. Ioana, Orbit inequivalent actions for groups containing a copy of $\mathbb{F}_2$                                                              ,  {\it Inventiones mathematicae}, {\bf 185} (2011), 55--73.


 \bibitem[K95]{K95} A.S. Kechris, Classical Descriptive Set Theory,  {\it Springer},  (1995). 
 
 
 \bibitem[K10]{K10} A.S. Kechris, Global Aspects of Ergodic Group Actions,  {\it Amer. Math. Soc.},  (2010). 
 
 
 \bibitem[K12]{K12} A.S. Kechris, Weak containment in the space of actions of a free group,  {\it Israel J. Math.}, {\bf 189} (2012), 461--507.

\bibitem[LM15]{LM15} F. Le Ma\^{i}tre, Highly faithful actions and dense free subgroups in full groups, {\it arXiv: 1509.03584v2}.


 \bibitem[LS77]{LS77} R. C. Lyndon and P. E. Schupp, Combinatorial Group Theory,  {\it Springer},  (1977).


\bibitem[O15]{O15} D. Osin, Invariant random subgroups of groups acting on hyperbolic spaces, {\it arXiv: 1510.07710v3}.


\bibitem[SZ94]{SZ94}G. Stuck and R.J. Zimmer, Stabilizers for ergodic actions of higher rank groups,  {\it Ann. of Math.}, {\bf 139(3)} (1994), 723--747. 



\bibitem[TT-D14]{TT-D14} S. Thomas and R.D. Tucker-Drob, Invariant random subgroups of strictly diagonal limits of finite symmetric groups,  {\it Bull. London. Math. Soc.}, {\bf 46(5)} (2014), 1007--1020. 

\bibitem[TT-D18]{TT-D18} S. Thomas and R.D. Tucker-Drob, Invariant random subgroups of inductive limits of finite alternating groups, {\it J. Algebra}, {\bf 503} (2018), 474--533.


 \bibitem[T-D15a]{T-D15a} R.D. Tucker-Drob, Mixing actions of countable groups are almost free,  {\it Proc. Amer. Math. Soc.}, {\bf 143} (2015), 5227--5232. 

 \bibitem[T-D15b]{T-D15b} R.D. Tucker-Drob, Weak equivalence and non-classifiability of measure preserving actions,  {\it Ergod. Th. {\&} Dynam. Sys.}, {\bf 35} (2015), 293--336. 

 
\end{thebibliography}
\end{document}